\documentclass{amsart}
\usepackage{amssymb, latexsym}

\newcommand{\beqa}{\begin{eqnarray*}}
\newcommand{\eeqa}{\end{eqnarray*}}
\newcommand{\beqn}{\begin{eqnarray}}
\newcommand{\eeqn}{\end{eqnarray}}

\newcommand{\iy}{\infty}

\newcommand{\lt}{\left}
\newcommand{\rt}{\right}

\newcommand{\bQ}{\mathbb Q}

\newcommand{\C}{\mathbb C}

\newcommand{\R}{\mathbb R}

\newcommand{\N}{\mathbb N}

\newcommand{\Ha}{\mathbb H}

\newcommand{\mcH}{\mathcal H}
\newcommand{\mcB}{\mathcal B}
\newcommand{\mcC}{\mathcal C}
\newcommand{\mcD}{\mathcal D}
\newcommand{\mcE}{\mathcal E}

\newcommand{\mcZ}{\mathcal Z}

\newcommand{\mfB}{\mathfrak B}

\newcommand{\mfM}{\mathfrak M}

\newcommand{\f}{\frac}
\newcommand{\tf}{\tfrac}

\newcommand{\al}{\alpha}

\newcommand{\G}{\Gamma}

\newcommand{\e}{\varepsilon}
\newcommand{\ph}{\phi}

\newcommand{\De}{\Delta}
\newcommand{\la}{\lambda}

\newcommand{\s}{\sigma}

\newcommand{\lb}{\label}
\newcommand{\rf}{\ref}
\newcounter{cnt1}
\newcounter{cnt2}
\newcounter{cnt3}
\newcommand{\blr}{\begin{list}{$($\roman{cnt1}$)$}
 {\usecounter{cnt1} \setlength{\topsep}{0pt}
 \setlength{\itemsep}{0pt}}}
\newcommand{\bla}{\begin{list}{$($\alph{cnt2}$)$}
 {\usecounter{cnt2} \setlength{\topsep}{0pt}
 \setlength{\itemsep}{0pt}}}
\newcommand{\bln}{\begin{list}{$($\arabic{cnt3}$)$}
 {\usecounter{cnt3} \setlength{\topsep}{0pt}
 \setlength{\itemsep}{0pt}}}
\newcommand{\el}{\end{list}}
\newtheorem{thm}{Theorem}[section]
\newtheorem{lem}[thm]{Lemma}
\newtheorem{cor}[thm]{Corollary}

\newtheorem{Def}[thm]{Definition}

\newtheorem{rem}[thm]{Remark}
\newcommand{\Rem}{\begin{rem} \rm}
\newcommand{\bdfn}{\begin{Def} \rm}
\newcommand{\edfn}{\end{Def}}
\newcommand{\tx}{\text}

\newcommand{\ba}{\begin{array}}
\newcommand{\ea}{\end{array}}

\numberwithin{equation}{section}
\date{}
\begin{document}

\title{Banach spaces for the Schwartz distributions}
\author[Gill]{Tepper L. Gill}
\address[Tepper L. Gill]{Department of Electrical \& Computer Engineering, and Mathematics \\Howard University\\
Washington DC 20059 \\ USA, {\it E-mail~:} {\tt tgill@howard.edu}}
\date{}
\subjclass{Primary (46) Secondary(47) }

\maketitle
\begin{abstract} This paper is a survey of a new family of Banach spaces $\mcB$ that provide the same structure for the Henstock-Kurzweil (HK) integrable functions as the $L^p$ spaces provide for the Lebesgue integrable functions.  These spaces also contain the wide sense Denjoy integrable functions.  They were first use to provide the foundations for the Feynman formulation of quantum mechanics. It has recently been observed that these spaces contain the test functions $\mcD$ as a continuous dense embedding. Thus, by the Hahn-Banach theorem, $\mcD' \subset \mcB'$.  

A new family that extend the space of functions of bounded mean oscillation $BMO[\R^n]$, to include the HK-integrable functions are also introduced.   
\end{abstract}
\section{{Introduction}}
Since the work of Henstock \cite{HS1} and Kurzweil \cite{KW},  the most important  finitely additive measure is the one generated by the Henstock-Kurzweil integral (HK-integral). It generalizes the Lebesgue, Bochner and Pettis integrals.  
The HK-integral is equivalent to the  Denjoy and Perron integrals.  However, it is much easier to understand (and learn) compared to the these and the Lebesgue integral.  It provides useful variants of the same theorems that have made the Lebesgue integral so important.   We assume that the reader is acquainted with this integral, but more detail can be found in Gill and Zachary \cite{GZ}.  (For different perspectives, see Gordon \cite{GO}, Henstock \cite{HS}, Kurzweil \cite{KW}, or Pfeffer \cite{PF}.)  

The most important factor preventing the widespread use of the HK-integral in mathematics, engineering  and physics is the lack of a Banach space structure comparable to the $L^p$ spaces for the Lebesgue integral.  The purpose of this paper is to provide a survey of some new classes of Banach spaces, which have this property and some with other interesting properties, but all contain the HK-integrable functions.
 
The first two classes are the ${KS}^p$ and the ${SD}^p$ spaces, $1 \leq p \leq \infty$. These are all separable spaces that contain the corresponding ${L}^p$ spaces as dense, continuous, compact embeddings.   We have recently discovered that these two classes also contain the test functions $\mcD$ as a continuous dense embedding.  This implies the each dual space contains the Schwartz distributions.  The family of  ${SD}^p$ spaces also have the remarkable property that ${\left\| {{D^\alpha }f} \right\|_{SD}} = {\left\| f \right\|_{SD}}$.

The other main classes of spaces ${\mcZ}^p$ and the ${\mcZ^{-p}}$ spaces, $1 \leq p \leq \infty$, are related to the space of functions of bounded mean oscillation, $BMO$.  We also introduce an extended version of this space, which we call the space of functions of weak bounded mean oscillation, $BMO^w$.  

We provide a few applications of the first two families of spaces, which either provide simpler solutions to old problems or solve open problems.

The main tool for the work in this paper had its beginnings in 1965, when Gross \cite{G} proved that every real separable Banach space contains a separable Hilbert space as a continuous dense embedding, which is the support of a Gaussian measure.  This was a generalization of Wiener's theory, which used the (densely embedded Hilbert) Sobolev space $\Ha_0^1[0,1] \subset \C_0[0,1]$. In 1972, Kuelbs \cite{KB} generalized Gross' theorem to include the Hilbert space rigging $\Ha_0^1[0,1] \subset \C_0[0,1] \subset L^2[0,1]$.  For our purposes, a general version of this theorem can be stated as:
\begin{thm}{\lb{GK}}{\rm{(Gross-Kuelbs)}}Let $\mcB$ be a separable Banach space. Then there exist separable Hilbert spaces $\mcH_1, \mcH_2$  and a positive trace class operator $T_{12}$ defined on $\mcH_2$ such that $\mcH_1\subset \mcB \subset \mcH_2$ all as continuous dense embeddings, with $\left( {T_{12}^{1/2} u,\;T_{12}^{1/2} v} \right)_1  = \left( {u,\;v} \right)_2$ and  $\left( {T_{12}^{ - 1/2} u,\;T_{12}^{ - 1/2} v} \right)_2  = \left( {u,\;v} \right)_1$.  
\end{thm}
A proof can be found in \cite{GZ}.   The space $\mcH_1$ is part of the abstract Wiener space method for extending Wiener measure to separable Banach spaces. The space $\mcH_2$ is a major tool in the construction Banach spaces for HK-integable functions.  (We call it the natural Hilbert space for $\mcB$.) 
\subsection{ Summary}
In Section 2, we establish a number of background results to make the paper self contained. The major result is Theorem \rf{5} (and Corollary \rf{be}).  It allows us to construct the path integral in the manner originally suggested by Feynman.  In Section 3, after a few examples, we construct the KS-spaces and derive some of their important properties. In Section 4, we construct the SD-spaces and discuss their properties.  In Section 5 we discuss the family of spaces related to the functions of bounded mean oscillation.   

In Section 6, we  give a few applications.  The first application uses ${KS}^2$ to provide a simple solution to the generator (with unbounded coefficients) problem for Markov processes.  The second application uses ${KS}^2$ and Corollary \rf{be} to construct the Feynman path integral. The third application uses $SD^2$ to provide the best possible a priori bound for the nonlinear term of the Navier-Stokes equation.   
\section{Background}
In this section, we provide some background results, which are required in the paper.  
Let $\mcB$ be a Banach space, with dual $\mcB^*$. 
\subsection{ Bounded Operator Extension}
We are interested in the problem of operator extensions from $\mcB$ to $\mcH(=\mcH_2)$.   It is not hard to see that, since $\mcB$ is a continuous dense embedding in $\mcH$, every closed densely defined linear operator on $\mcB$ has a  closed densely defined extension to $\mcH$,  $\mcC[\mcB] \xrightarrow{{{\text{ext}}}} \mcC[\mcH]$ (see Theorem \rf{5} for a proof).  In this section, we show that this also holds for bounded linear operators, $L[\mcB] \xrightarrow{{{\text{ext}}}} L[\mcH]$. This important result depends on the following theorem by Lax \cite{L}.  It is not well known, so we include a proof.
\begin{thm}{\lb{L}}{\bf{\rm{(Lax's Theorem)}}} Let $\mcB$ be a separable Banach space continuously and densely embedded in a Hilbert space $\mcH$ and let $T$ be a bounded linear operator on $\mcB$ which is symmetric with respect to the inner product of $\mcH$ (i.e., $(Tu,v)_\mcH =(u,Tv)_\mcH$ for all $u,v \in \mcB$).  Then:
\begin{enumerate}
\item The operator $T$ is bounded with respect to the $\mcH$ norm and 
\[
\left\| {T^* T} \right\|_\mcH  = \left\| T\right\|_\mcH^2  \leqslant k\left\| T \right\|_\mcB^2,
\] 
where $k$ is a positive constant.
\item  The spectrum of $T$ relative to $\mcH$ is a subset of the spectrum of $T$ relative to $\mcB$.
\item The point spectrum of $T$ relative to $\mcH$ is a equal to the point spectrum of $T$ relative to $\mcB$.
\end{enumerate} 
\end{thm}
\begin{proof} To prove (1),  let $u  \in {\mathcal{B}}$
 and, without loss, we can assume that $k = 1$ and $\left\| u  \right\|_\mathcal{H}  = 1$.  Since $T$ is selfadjoint, 
\[
\left\| {Tu } \right\|_\mathcal{H}^2  = \left( {Tu ,Tu } \right) = \left( {u ,T^2 u } \right) \leqslant \left\| u  \right\|_\mathcal{H} \left\| {T^2 u } \right\|_\mathcal{H}  = \left\| {T^2 u } \right\|_\mathcal{H}. 
\]
Thus, we have $\left\| {Tu } \right\|_\mathcal{H}^4  \leqslant \left\| {T^4 u } \right\|_\mathcal{H}$, so it is easy to see that $
\left\| {Tu } \right\|_\mathcal{H}^{2n}  \leqslant \left\| {T^{2n} u } \right\|_\mathcal{H}$ for all $n$.  It follows that: 
\[
\begin{gathered}
  \left\| {Tu } \right\|_\mathcal{H}  \leqslant (\left\| {T^{2n} u } \right\|_\mathcal{H} )^{1/2n}  \leqslant (\left\| {T^{2n} u } \right\|_\mathcal{B} )^{1/2n}  \hfill \\
  {\text{          }} \leqslant (\left\| {T^{2n} } \right\|_\mathcal{B} )^{1/2n} (\left\| u  \right\|_\mathcal{B} )^{1/2n}  \leqslant \left\| T \right\|_\mathcal{B} (\left\| u  \right\|_\mathcal{B} )^{1/2n}.  \hfill \\ 
\end{gathered} 
\]
Letting $n \to \infty $, we get that $\left\| {Tu } \right\|_\mathcal{H}  \leqslant \left\| T \right\|_{\mathcal{B}} $ for $u$ in a dense set of the unit ball of $\mathcal{H}$.  It follows that 
\[
{\left\| T \right\|_{\mathcal{H}}} = \mathop {\sup }\limits_{{{\left\| u \right\|}_{\mathcal{H}}} \leqslant 1} {\left\| {Tu} \right\|_{\mathcal{H}}} \leqslant {\left\| T \right\|_{\mathcal{B}}}.
\]

To prove (2), suppose $ \la_0 \notin \s_T$, the spectrum of $T$ over $\mcB$ so that $T-\la_0 I$ has a bounded inverse $S$ on $\mcB$.  Since $T$ is symmetric on $\mcH$, so is $S$.  It follows from (1) that $T$ and $S$ extend to bounded linear operators $\bar{T}$ and $\bar{S}$ on $\mcH$.  We also see that $ \la_0 \notin \s_{\bar{T}}$.  It follows from this that $\bar{T}$ has inverse $\bar{S}$ and the spectrum of $\bar{T}$ on $\mcH$ is a subset of the spectrum of $T$ on $\mcB$ (i.e., $\s_{\bar{T}} \subset \s_T$).

To prove (3),  suppose that $\la \in \s_p$, the point spectrum of $T$, so that $T-\la I$ has a finite dimensional null space $N$ and $dim{N} =dim\{\mcB/J\}$, where $J=(T-\la I)(\mcB)$.

Since $T$ is symmetric, every vector in $J$ is orthogonal to $N$.  Conversely, from $dim N = dim\{\mcB/J\}$ we see that $J$ contains all vectors that are orthogonal to $N$.  It follows that, $(T-\la I)$ is a one-to-one, onto mapping of $J \to  J$, so that $T - \la I=S$ has an inverse on  $J$, which is bounded (on $J$) by the Closed Graph Theorem.  It follows that the extension $\bar{S}$ of $S$ to the closure of $J, \bar{J}$ in $\mcH$ is bounded  on $ \bar{J}$.  This means that $(\bar{T}-\la I)$ is the orthogonal compliment of $N$ over $\mcH$, so that $\la$ belongs to the point spectrum of $\bar{T}$ on $\mcH$ and the null space of $(\bar{T}-\la I)$ over $\mcH$ is $N$.  It follows that the point spectrum of $T$ is unchanged on extension to $\mcH$.
\end{proof} 
Let  $\mcH$ be the natural Hilbert space for $\mcB$,  let ${\bf J}$ be the standard linear mapping from ${\mcH}  \to {\mcH}^* $ and let ${\bf J}_\mcB$ be its restriction to $\mcB$.  Since $\mcB$ is a continuous dense embedding in $\mcH$, ${\bf J}_\mcB$ is a (conjugate) isometric isomorphism of $\mcB$ onto ${\bf J}_\mcB(\mcB) \subset \mcB^*$. 
\begin{Def} If $u \in \mcB$, set  $u_h={\bf J}_\mcB(u)$ and define
\[
\mcB_h^* = \left\{ {{u_h} \in {\mcB^*}:u \in \mcB} \right\},
\]
so that $\left\langle {u,{u_h}} \right\rangle =(u,u)_\mcH  = \left\| u \right\|_\mcH^2$.
It is clear from our construction of $\mcB_h^*$ that the mapping taking $\mcB \to \mcB_h^*$ is a (conjugate)  isometric isomorphism. We call $\mcB_h^*$ the h-representation for $\mcB$ in ${\mcB^*}$.  
\end{Def}
Its easy to prove that following result.
\begin{thm} If $\mcB$ is a reflexive Banach space.  Then  $\mcB_h^*$ is bijectively related to $\mcB^*$.   
\end{thm}
\begin{rem}In general, the embedding of $\mcB_h^*$ is a proper subspace of $\mcB^*$.  
\end{rem}
We can now state and prove the following fundamental theorem.
\begin{thm}\lb{5} Let $\mcB$ be a separable Banach space. If $A \in \mcC[\mcB]$, then there is a unique operator $A^* \in \mcC[\mcB]$ satisfying:  
\begin{enumerate}
\item $(aA)^* ={\bar a} A^*$,
\item $A^{**} =A$,
\item $(A^* +B^*)= A^* + B^*$, 
\item $(AB)^*= B^*A^*$ on $D(A^*) \bigcap D(B^*)$, 
\item $(A^*A)^*= A^*A$ on $D(A^*A)$ (self adjoint),  
\item if $A \in L[\mcB]$, then $\lt\|A^*A\rt\|_\mcH \le k \lt\|A^*A\rt\|_\mcB$ and
\item if $A \in L[\mcB]$, then $\lt\|A^*A\rt\|_\mcB \le c \lt\|A\rt\|_\mcB^2$, for some constant $c$. 
\end{enumerate}
\end{thm}
\begin{proof}
Recall that ${\bf J}$ is the natural linear mapping from $\mcH_2={\mcH}  \to {\mcH}^* $ and ${\bf J}_\mcB$ is the restriction of ${\bf J}$ to $\mcB$, so that $\bf{J}_\mcB(\mcB)=\mcB_h^*$.    If ${A}  \in \mcC[{\mcB}]$, then $A': \mcB^* \to \mcB^*$ and  $ {{A'} {\bf J}_{\mcB} } : \mcB \to \mcH^*$. Since $\mcB$ is dense in $\mcH, \ \mcB_h^*$ is dense in $\mcH^*$.  It follows that ${A'} {\bf J}_\mcB$ is a closed densely define operator on $\mcH^*, \ {\bf J}_\mcB^{ - 1} {A'} {\bf J}_\mcB :{\mcB}  \to {\mcB}$ is a closed and densely defined linear operator on $\mcB$. We define ${A}^ {*}   = [ {{\bf J}_\mcB^{ - 1} {A'} {\bf J}_\mcB } ] \in \mcC[\mcB]$. If  ${A}  \in L[{\mcB}]$, ${A}^ {*} =  {{\bf J}_\mcB^{ - 1} {A'} {\bf J} }_{\mcB}$ is defined on all of $\mcB$ so that, by the Closed Graph Theorem, ${A}^ {*} \in L[\mcB]$.  
The proofs of (1)-(3) are straight forward.  To prove (4), let $u \in D(A^*) \bigcap D(B^*)$, then
\beqn\lb{5: prod}
\begin{gathered}
  {\left( {BA} \right)^*} u = [{\mathbf{J}}_\mcB^{ - 1}{\left( {BA} \right)^\prime }{\mathbf{J}}_\mcB] u = [{\mathbf{J}}_\mcB^{ - 1}{A^\prime }B'{\mathbf{J}}_\mcB ] u \hfill \\
   = \left[ {{\mathbf{J}}_\mcB^{ - 1}{A^\prime }{\mathbf{J}}_\mcB} \right]\left[ {{\mathbf{J}}_\mcB^{ - 1}{B^\prime }{\mathbf{J}}_\mcB} \right] u = {A^*}{B^*} u. \hfill \\ 
\end{gathered} 
\eeqn
If we replace $B$ by $A^*$ in equation (\rf{5: prod}), noting that $A^{**}=A$, we also see that $(A^*A)^*=A^*A$, proving (5).  To prove (6), we first see that:
\[
\left\langle {{A^*}Av,{{\mathbf{J}}_\mcB}(u)} \right\rangle= \left\langle {{A^*}Av, u_h} \right\rangle = {\left( {{A^*}Av,u} \right)_\mcH} = {\left( {v,{A^*}Au} \right)_\mcH},
\]
so that $A^*A$ is symmetric.  Thus, by Theorem \rf{L} (Lax), $A^*A$ has a bounded extension to $\mcH$ and $\left\| {A^* A} \right\|_\mcH=\left\| {A} \right\|_\mcH^2 \leqslant k\left\| {A^* A} \right\|_\mcB $, where $k$ is a positive constant.  We also have that 
\beqn\lb{6: prod}
{\left\| {{A^*}A} \right\|_\mcB} \leqslant {\left\| {{A^*}} \right\|_\mcB}{\left\| A \right\|_\mcB} \leqslant {\left\| {{{\mathbf{J}}_{\mcB}}} \right\|_{\mcB^*}}{\left\| {{\mathbf{J}}_\mcB^{ - 1}} \right\|_\mcB}\left\| A' \right\|_{\mcB^*}\left\| A \right\|_\mcB= c\left\| A \right\|_\mcB^2,
\eeqn
proving (7).  It also follows that 
\beqn\lb{7: prod}
\left\|  {A} \right\|_\mcH \le \sqrt{ck} \left\| {A} \right\|_\mcB.
\eeqn
\end{proof}
If $c=1$ and equality holds in (\rf{6: prod})  for all $A \in L[\mcB]$, then $L[\mcB]$ is a $C^*$-algebra.  In this case, $\mcB$ is a Hilbert space.  Thus, in general  the inequality in (\rf{6: prod}) is strict.  
From (\rf{7: prod}), we see the following.  
\begin{cor}\lb{be} Let $\mcB$ be a separable Banach space. If $A \in L[\mcB]$, then there is a unique operator $\bar A \in L[\mcH]$ (i.e., $L[\mcB] \xrightarrow{{{\text{ext}}}} L[\mcH]$). 
\end{cor}
\begin{thm}{\rm{(Polar Representation)}} Let $\mcB$ be a  separable Banach space. If $A \in \C[\mcB]$, then there exists a partial isometry $U$ and a self-adjoint operator $T, \ T=T^*$, with $D(T)=D(A)$ and $A=UT$.  
\end{thm}
\begin{proof} Let $\bar A$  be the closed densely defined extension of $A$ to ${\mathcal H}$.  On  ${\mathcal{H}}$,  $\bar T^2={\bar A}^* \bar A$ is self-adjoint and there exist a unique partial isometry $\bar U$, with $\bar A=\bar U \bar T$.  Thus, the restriction to ${\mathcal{B}}$ gives us $A=UT$ and $U$ is a partial isometry on $\mcB$.  (It is easy to check that $A^*A=T^2$.)
\end{proof}
\section{The Kuelbs-Steadman $KS^p$ Spaces}
\subsection{Special Constructions}
Our first construction is based on an extension of a norm due to Alexiewicz \cite{AL}.  we first recall that the HK-integral is equivalent to the strict Denjoy integral (see Henstock \cite{HS} or Pfeffer \cite{PF}).  In the one-dimensional case, Alexiewicz \cite{AL} has shown that the class $D({\mathbb{R}})$, of Denjoy integrable functions, can be normed in the following manner:  for $f \in D({\mathbb{R}})$, define $\left\| f \right\|_{A_1}$ by
\[
\left\| f \right\|_{A_1}  = \sup _s \left| {\int_{ - \infty }^s {f(r)d\la(r)} } \right|.
\]
It is clear that this is a norm, and it is known that $D({\mathbb{R}})$ is not complete (see Alexiewicz \cite{AL}). If we replace ${\mathbb{R}}$ by ${\mathbb{R}}^n$, let $[a_i, b_i] \subset \bar{\R}=[-\iy, \iy]$, and define $[{\bf a}, {\bf b}] \in {\bar \R}^n$ by   $[{\bf a}, {\bf b}]=\prod_{k=1}^n{[a_i,b_i]}$.   If $f \in D({\mathbb{R}}^n )$ we define the norm of $f$ by (see \cite{OS}): 
\beqn
{\left\| f \right\|_{A_n}} = \mathop {\sup }\limits_{{x_1} \cdots {x_n}} \left| {\int_{ - \infty }^{{x_1}} { \cdots \int_{ - \infty }^{{x_n}} {f({\mathbf{y}})d{\lambda _n}({\mathbf{y}})} } } \right|.
\eeqn
\begin{lem}If $HK[\R^n]$ is the completion of $D({\mathbb{R}}^n )$ in the above norm, then $L_{\tx{Loc}}^1[\R^n] \subset HK[\R^n]$, as a continuous embedding.
\end{lem}
The completion of $\C_0[\R]$ in the $L^1$ norm leads to all absolutely integrable functions (i.e., the norm limit of $L^1$ functions is an $L^1$ function).  The completion of $\C_0[\R]$ in the Alexiewicz norm leads to $HK[\R^n]$, but the class of HK-integrable functions is a  proper subset of $HK[\R^n]$. This is the case for all suggested norms and H{\"o}nig has conjectured  that there is no ``natural norm" for the HK-integrable functions (see \cite{HO}). He suggests that the following is the best we can hope for:
\begin{lem} {\tx{(H{\"o}nig)}} Let $HK[a,b]$ be the space of HK-integrable functions on $[a,b] \subset \R$ and let $\C_a[a,b]$ be the continuous functions $f$ on $[a,b]$, with $f(a)=0$.
If $\left\| \cdot \right\|_{A_1}$ is the Alexiewicz norm:
\[
\left\| f \right\|_{A_1}  = \sup _s \left| {\int_{ a }^s {f(r)d\la(r)} } \right|,
\]
then the completion of $HK[a,b]$ is the set of all distributions that are weak derivatives of functions in $\C_a[a,b]$.
\end{lem}
Talvila \cite{TA} independently obtained the same result for $\R$ and used it to motivate his definition of a distributional integral. 
\begin{Def}  
If there is a continuous function $F(x)$ with real limits at infinity such that $F'(x)= f(x)$ (weak derivative), then the distributional integral of $f(x)$ is defined to be $D\int_{ - \infty }^\infty  {f(x)dx}  = F(\infty ) - F( - \infty )$.  
\end{Def} 
Talvila shows that the Alexiewicz norm leads to a Banach space of integrable distributions that is isometrically isomorphic to the space of continuous functions on the extended real line with uniform norm. He also shows that the dual space can be identified with the space of functions of bounded variation.

The following two spaces are also closely related to $HK[\R^n]$.
\begin{thm} Let $\{u_i\}_{i=1}^\iy \subset \C_0^1[\R^n]$ be S-basis for $\mcB=L^1[\R^n]$ or an orthonormal basis for $\mcH= L^2[\R^n]$.  
\begin{enumerate}
\item Then, for $L^1[\R^n]$,
\[
{\left\| f \right\|_{{A_n^1 }}} = \mathop {\sup }\limits_i \left| {\int_{{\mathbb{R}^n}} {f\left( {\bf x} \right){u_i}\left({\bf x}  \right)d{\lambda _n}\left( {\bf x}  \right)} } \right|,
\] 
defines a weaker norm on $L^1[\R^n]$ and  $f\in L^1[\R^n] \Rightarrow {\left\| f \right\|_{{A_n^1}}} \le {\left\| f \right\|_{{1}}}$, and
\item If $\mcB_{A_n^1}$ is the completion of $\mcB$ in this norm, then $L_{\tx{Loc}}^1[\R^n] \subset \mcB_{A_n^1}$, as a continuous embedding. 
\item For $L^2[\R^n]$
\[
{\left\| f \right\|_{{A_n^2 }}} = \mathop {\sup }\limits_i \left| {\int_{{\mathbb{R}^n}} {f\left( {\bf x} \right){u_i}\left({\bf x}  \right)d{\lambda _n}\left( {\bf x}  \right)} } \right|,
\] 
defines a weaker norm on $L^2[\R^n]$ and  $f\in L^2[\R^n] \Rightarrow {\left\| f \right\|_{{A_n^2  }}} \le {\left\| f \right\|_{{2}}}$. 
\item If $\mcH_{A_n^2 }$ is the completion of $\mcH$ in this norm, then $L_{\tx{Loc}}^2[\R^n] \subset \mcH_{A_n^2 }$, as a continuous embedding. 
\end{enumerate}
\end{thm}
\begin{rem} 
The last two spaces are close but not the same.  A proof is actually required to show that $\mcB_{A_n^1}$ and  $\mcH_{A_n^2 }$ contains the HK-integrable functions.  The condition  $\{u_i\}_{i=1}^\iy \subset \C_0^1[\R^n]$ is sufficient and the proof is the same as for Lemma \rf{8} (see Section 4.3).
\end{rem} 
\subsection{General Construction}
For our first general construction,  fix $n$, and let $\mathbb{Q}^n$ be the set $
\left\{ {{\mathbf{x}} = (x_1 ,x_2  \cdots ,x_n ) \in {\mathbb{R}}^n } \right\}$ such that $
x_i$ is rational for each $i$.  Since this is a countable dense set in ${\mathbb{R}}^n $, we can arrange it as $\mathbb{Q}^n  = \left\{ {{\mathbf{x}}_1, {\mathbf{x}}_2, {\mathbf{x}}_3,  \cdots } \right\}$.  For each $l$ and $i$, let ${\mathbf{B}}_l ({\mathbf{x}}_i ) $ be the closed cube centered at ${\mathbf{x}}_i$,  with sides parallel to the coordinate axes and edge $e_l=\tfrac{1 }{2^{l}\sqrt{n}}, l \in {\mathbb{N}}$.  Now choose the  natural  order which maps $\mathbb{N} \times \mathbb{N}$ bijectively to $\mathbb{N}$, and let $\left\{ {{\mathbf{B}}_{k} ,\;k \in \mathbb{N}}\right\}$ be the resulting set  of (all) closed cubes $
\{ {\mathbf{B}}_l ({\mathbf{x}}_i )\;\left| {(l,i) \in \mathbb{N} \times \mathbb{N}\} } \right.
$
centered at a point in $\mathbb{Q}^n $.  Let ${\mathcal{E}}_k ({\mathbf{x}})$ be the characteristic function of ${\mathbf{B}}_k $, so that ${\mathcal{E}}_k ({\mathbf{x}})$ is in ${{L}}^p [{\mathbb{R}}^n ] \cap {{L}}^\infty  [{\mathbb{R}}^n ] $ for $1 \le p < \infty$.  Define $F_{k} (\; \cdot \;)$ on $
{{L}}^1 [{\mathbb{R}}^n ] $ by
 \beqn
F_{k} (f) = \int_{{\mathbb{R}}^n } {{\mathcal{E}}_{k} ({\mathbf{x}})f({\mathbf{x}})d\la_n({\mathbf{x}}) }. 
\eeqn
It is clear that $F_{k} (\; \cdot \;)$ is a bounded linear functional on ${{L}}^p [{\mathbb{R}}^n ] $ for each ${k}$, $\left\| {F_{k} } \right\|_\infty   \le 1$ and, if $F_k (f) = 0$ for all ${k}$, $f = 0$ so that $\left\{ {F_{k} } \right\}$ is fundamental on ${{L}}^p [{\mathbb{R}}^n ] $ for $1 \le p \le \infty$ .
Fix ${t_{k}}> 0 $ such that ${\sum\nolimits_{k = 1}^\infty  {t_k}}=1$  and define a measure $d{\mathbf{P}}  ({\mathbf{x}},{\mathbf{y}})$ on ${\mathbb{R}}^n \, \times {\mathbb{R}}^n $ by: 
\[
d{\mathbf{P}}  ({\mathbf{x}},{\mathbf{y}}) = \left[ {\sum\nolimits_{k = 1}^\infty  {t_k {\mathcal{E}}_k ({\mathbf{x}}){\mathcal{E}}_k ({\mathbf{y}})} } \right]d\la_n({\mathbf{x}})d\la_n({\mathbf{y}}).
\]
We first construct our Hilbert space.  Define an inner product $\left( {\; \cdot \;} \right) $ on ${{L}}^1 [{\mathbb{R}}^n ] $ by
\beqn
\begin{gathered}
 \left( {f,g} \right) = \int_{\mathbb{R}^n  \times \mathbb{R}^n } {f({\mathbf{x}})g({\mathbf{y}})^ *  d{\mathbf{P}}  ({\mathbf{x}},{\mathbf{y}})}  \hfill \\
{\text{             }} = \sum\nolimits_{k = 1}^\infty  {t_k } \left[ {\int_{\mathbb{R}^n } {{\mathcal{E}}_k ({\mathbf{x}})f({\mathbf{x}})d\la_n({\mathbf{x}})} } \right]\left[ {\int_{\mathbb{R}^n } {{\mathcal{E}}_k ({\mathbf{y}})g({\mathbf{y}})d\la_n({\mathbf{y}})} } \right]^ *.   \hfill \\ 
\end{gathered} 
\eeqn
We use a particular choice of $t_k$ in Gill and Zachary [GZ], which  is suggested by physical analysis in another context.  We call the completion of ${{L}}^1 [{\mathbb{R}}^n ] $, with the above inner product, the Kuelbs-Steadman space, ${K}{S}^2 [{\mathbb{R}}^n ] $.  Following suggestions of Gill and Zachary, Steadman [ST] constructed this space by adapting an approach developed by Kuelbs [KB] for other purposes.  Her interest was in showing that ${{L}}^1 [{\mathbb{R}}^n ] $ can be densely and continuously embedded in a Hilbert space which contains the HK-integrable functions.  To see that this is the case, let $f \in D[{\mathbb{R}}^n ] $, then:
\[
\left\| f \right\|_{{KS}^2}^2  = \sum\nolimits_{k = 1}^\infty  {t_k } \left| {\int_{\mathbb{R}^n } {{\mathcal{E}}_k ({\mathbf{x}})f({\mathbf{x}})d\la_n({\mathbf{x}})} } \right|^2  \leqslant \sup _k \left| {\int_{\mathbb{R}^n } {{\mathcal{E}}_k ({\mathbf{x}})f({\mathbf{x}})d\la_n({\mathbf{x}})} } \right|^2  \leqslant \left\| f \right\|_{A_n}^2, 
\]
so $f \in {K}{S}^2 [{\mathbb{R}}^n ] $.  
\begin{thm} For each $p,\;1 \leqslant p \leqslant \infty,$
 ${K}{S}^2 [{\mathbb{R}}^n ] \supset {{L}}^p [{\mathbb{R}}^n ]$
as a dense subspace.
\end{thm}
\begin{proof} By construction, ${K}{S}^2 [{\mathbb{R}}^n ]$
contains ${{L}}^1 [{\mathbb{R}}^n ]$ densely, so we need only show that 	
${K}{S}^2 [{\mathbb{R}}^n ] \supset {{L}}^q [{\mathbb{R}}^n ]$
for $q \ne 1$.  If $f \in {{L}}^q [{\mathbb{R}}^n ]$ and $q < \infty $, we have 
\[
\begin{gathered}
 \left\| f \right\|_{{KS}^2}  = \left[ {\sum\nolimits_{k = 1}^\infty  {t_k } \left| {\int_{{\mathbb{R}}^n } {{\mathcal{E}}_k ({\mathbf{x}})f({\mathbf{x}})d\la_n({\mathbf{x}})} } \right|^{\frac{{2q}}{q}} } \right]^{1/2}  \hfill \\
{\text{       }} \leqslant \left[ {\sum\nolimits_{k = 1}^\infty  {t_k } \left( {\int_{{\mathbb{R}}^n } {{\mathcal{E}}_k ({\mathbf{x}})\left| {f({\mathbf{x}})} \right|^q d\la_n({\mathbf{x}})} } \right)^{\frac{2}{q}} } \right]^{1/2}  \hfill \\
{\text{      }} \leqslant \sup _k \left( {\int_{{\mathbb{R}}^n } {{\mathcal{E}}_k ({\mathbf{x}})\left| {f({\mathbf{x}})} \right|^q d\la_n({\mathbf{x}})} } \right)^{\frac{1}
{q}}  \leqslant \left\| f \right\|_q . \hfill \\ 
\end{gathered} 
\]
Hence, $f \in {K}{S}^2 [{\mathbb{R}}^n ] $.  For $q = \infty $, first note that $ vol({\mathbf{B}}_k )^2 \le \left[ {\frac{1}
{{2\sqrt n }}} \right]^{2n}$, so we have 
\[
\begin{gathered}
  \left\| f \right\|_{{KS}^2}  = \left[ {\sum\nolimits_{k = 1}^\infty  {t_k } \left| {\int_{{\mathbf{R}}^n } {{\mathcal{E}}_k ({\mathbf{x}})f({\mathbf{x}})d\la_n({\mathbf{x}})} } \right|^2 } \right]^{1/2}  \hfill \\
  {\text{       }} \leqslant \left[ {\left[ {\sum\nolimits_{k = 1}^\infty  {t_k [vol({\mathbf{B}}_k )]^2 } } \right][ess\sup \left| f \right|]^2 } \right]^{1/2}  \leqslant {\left[ {\frac{1}
{{2\sqrt n }}} \right]^{n}}\left\| f \right\|_\infty  . \hfill \\ 
\end{gathered} 
\]
Thus $f \in {K}{S}^2 [{\mathbb{R}}^n ] $, and ${{L}}^\infty  [{\mathbb{R}}^n ] \subset {K}{S}^2 [{\mathbb{R}}^n ]$.
\end{proof}
Before proceeding to additional study, we construct the ${K}{S}^p [{\mathbb{R}}^n ]$ spaces, for $1 \le p \le \iy$. 

To construct ${K}{S}^p [{\mathbb{R}}^n ]$ for all $p$ and  for $f \in {L}^p$, define:
\[
\left\| f \right\|_{{{KS}}^p }  = \left\{ {\begin{array}{*{20}c}
   {\left\{ {\sum\nolimits_{k = 1}^\infty  {t_k \left| {\int_{\mathbb{R}^n } { {\mathcal{E}}_k ({\mathbf{x}})f({\mathbf{x}})d\la_n({\mathbf{x}})} } \right|} ^p } \right\}^{1/p} ,1 \leqslant p < \infty},  \\
   {\sup _{k \geqslant 1} \left| {\int_{\mathbb{R}^n } { {\mathcal{E}}_k ({\mathbf{x}})f({\mathbf{x}})d\la_n({\mathbf{x}})} } \right|,p = \infty .}  \\

 \end{array} } \right.
\] 
It is easy to see that $\left\| \cdot \right\|_{{{KS}}^p }$ defines a norm on  ${L}^p$.  If ${{{KS}}^p }$ is the completion of ${L}^p$ with respect to this norm, we have:
\begin{thm} For each $q,\;1 \leqslant q \leqslant \infty,$
 ${K}{S}^p [{\mathbb{R}}^n ] \supset {{L}}^q [{\mathbb{R}}^n ]$ as a dense continuous embedding.
\end{thm}
\begin{proof} As in the previous theorem, by construction ${K}{S}^p [{\mathbb{R}}^n ]$
contains ${{L}}^p [{\mathbb{R}}^n ]$ densely, so we need only show that 	
${K}{S}^p [{\mathbb{R}}^n ] \supset {{L}}^q [{\mathbb{R}}^n ]$
for $q \ne p$.  First, suppose that $p< \infty$.  If $f \in {{L}}^q [{\mathbb{R}}^n ]$ and $q < \infty $, we have 
\[
\begin{gathered}
 \left\| f \right\|_{{KS}^p}  = \left[ {\sum\nolimits_{k = 1}^\infty  {t_k } \left| {\int_{{\mathbb{R}}^n } {{\mathcal{E}}_k ({\mathbf{x}})f({\mathbf{x}})d\la_n({\mathbf{x}})} } \right|^{\frac{{qp}}{q}} } \right]^{1/p}  \hfill \\
{\text{       }} \leqslant \left[ {\sum\nolimits_{k = 1}^\infty  {t_k } \left( {\int_{{\mathbb{R}}^n } {{\mathcal{E}}_k ({\mathbf{x}})\left| {f({\mathbf{x}})} \right|^q d\la_n({\mathbf{x}})} } \right)^{\frac{p}{q}} } \right]^{1/p}  \hfill \\
{\text{      }} \leqslant \sup _k \left( {\int_{{\mathbb{R}}^n } {{\mathcal{E}}_k ({\mathbf{x}})\left| {f({\mathbf{x}})} \right|^q d\la_n({\mathbf{x}})} } \right)^{\frac{1}
{q}}  \leqslant \left\| f \right\|_q . \hfill \\ 
\end{gathered} 
\]
Hence, $f \in {K}{S}^p [{\mathbb{R}}^n ] $.  For $q = \infty $, we have 
\[
\begin{gathered}
  \left\| f \right\|_{{KS}^p}  = \left[ {\sum\nolimits_{k = 1}^\infty  {t_k } \left| {\int_{{\mathbb{R}}^n } {{\mathcal{E}}_k ({\mathbf{x}})f({\mathbf{x}})d\la_n({\mathbf{x}})} } \right|^p } \right]^{1/p}  \hfill \\
  {\text{       }} \leqslant \left[ {\left[ {\sum\nolimits_{k = 1}^\infty  {t_k [vol({\mathbf{B}}_k )]^p } } \right][ess\sup \left| f \right|]^p } \right]^{1/p}  \leqslant M\left\| f \right\|_\infty  . \hfill \\ 
\end{gathered} 
\]
Thus $f \in {K}{S}^p [{\mathbb{R}}^n ] $, and ${{L}}^\infty  [{\mathbb{R}}^n ] \subset {K}{S}^p [{\mathbb{R}}^n ]$.   The case $p=\infty$ is obvious.
\end{proof} 
\begin{thm} For ${K}{S}^p$, $1\leq p \leq \infty$, we have: 
\begin{enumerate}
\item If $f,g \in {K}{S}^p$, then 
$
\left\| {f + g} \right\|_{{{KS}}^{{p}} }  \leqslant \left\| f \right\|_{{{KS}}^{{p}} }  + \left\| g \right\|_{{{KS}}^{{p}} }
$  (Minkowski inequality). 
\item If $K$ is a weakly compact subset of ${{L}^p}$, it is a compact subset of  ${K}{S}^p$.
\item If $1< p < \infty$, then ${K}{S}^p$ is uniformly convex.
\item If $1< p < \infty$ and $p^{ - 1}  + q^{ - 1}  = 1$, then the dual space of ${K}{S}^p$ is ${K}{S}^q$.
\item  ${K}{S}^{\infty} \subset {K}{S}^p$, for $1\leq p < \infty$.
\end{enumerate}
\end{thm}
\begin{proof}
The proof of (1) follows from the classical case for sums.  The proof of (2) follows from the fact that, if $\{f_m \}$ is any weakly convergent sequence in $K$ with limit $f$, then 
\[
\int_{\mathbb{R}^n } { {\mathcal{E}}_k ({\mathbf{x}})\left[ {f_m ({\mathbf{x}}) - f({\mathbf{x}})} \right]d\la_n({\mathbf{x}})}  \to 0
\]
for each $k$.  It follows that $\{f_m \}$ converges strongly to $f$ in ${K}{S}^p$.  

The proof of (3) follows from a modification of the proof of the Clarkson inequalities for $l^p$ norms.

In order to prove (4), observe that, for $p \ne 2, \; 1<p< \infty$,  the linear functional 
\[
L_g (f) = \left\| g \right\|_{{{KS}}^p }^{2 - p} \sum\nolimits_{k = 1}^\infty  {t_k \left| {\int_{\mathbb{R}^n } { {\mathcal{E}}_k ({\mathbf{x}})g({\mathbf{x}})d\la_n({\mathbf{x}})} } \right|} ^{p - 2} \int_{\mathbb{R}^n } { {\mathcal{E}}_k ({\mathbf{y}})f({\mathbf{y}})^* d\la_n({\mathbf{y}})}
\]
is a unique duality map on ${K}{S}^q$ for each $g \in  {K}{S}^p$ and that ${K}{S}^p$ is reflexive from (3). To  prove (5), note that $f \in {K}{S}^{\infty}$ implies that $\left| {\int_{\mathbb{R}^n } {\mathcal{E}_k ({\mathbf{x}})f({\mathbf{x}})d\la_n({\mathbf{x}})} } \right|$ is uniformly bounded for all $k$.  It follows that $\left| {\int_{\mathbb{R}^n } {\mathcal{E}_k ({\mathbf{x}})f({\mathbf{x}})d\la_n({\mathbf{x}})} } \right|^{p}$ is uniformly bounded for each $p,\;1 \le p < \infty$.  It is now clear from the definition of ${K}{S}^{\infty}$ that:
\[\left[ {\sum\nolimits_{k = 1}^\infty  {t_k \left| {\int_{\mathbb{R}^n } {\mathcal{E}_k ({\mathbf{x}})f({\mathbf{x}})d\la_n({\mathbf{x}})} } \right|^p } } \right]^{1/p}  \leqslant \left\| f \right\|_{{K}{S}^\infty  }  < \infty
\]
\end{proof}
Note that,  since ${{L}}^1 [{\mathbb{R}}^n] \subset {K}{S}^p [{\mathbb{R}}^n]$ and ${K}{S}^p [{\mathbb{R}}^n]$ is reflexive for $1 < p < \infty $, we see that the second dual $\left\{ {{{L}}^1 [{\mathbb{R}}^n]} \right\}^{**}  = \mathfrak{M}[{\mathbb{R}}^n] \subset {K}{S}^p [{\mathbb{R}}^n]$.   Recall that $\mathfrak{M}[{\mathbb{R}}^n]$ is the space of bounded finitely additive set functions defined on the Borel sets $\mathfrak{B}[{\mathbb{R}}^n]$.   

In many applications, it is convenient to formulate problems on one of the standard Sobolev spaces ${W}^{m,p} (\mathbb{R}^n)$.    We can easily see that ${{L}}^{p}_{\rm{loc}}(\mathbb{R}^n) \subset {KS}^q (\mathbb{R}^n), \; 1 \le q \le \infty$, for all $p, \; 1 \le p \le \infty$.   This means that ${KS}^q({\R}^n)$ contains a large class of distributions (see Adams \cite{A}). There is more:  
\begin{thm} For each $p, \; 1 \le p \le \infty$, the test functions $\mcD \subset {KS}^p (\mathbb{R}^n) $ as a continuous embedding. 
\end{thm}
\begin{proof} Since $KS^\iy(\R^n)$ is continuously embedded in ${KS}^p (\mathbb{R}^n), \; 1 \le q < \infty$, it suffices to prove the result for $KS^\iy(\R^n)$. Suppose that $\phi_j \to \phi$ in $\mcD[\R^n]$, so that there exist a compact set $K \subset \R^n$, containing the support of $\ph_j -\phi$ and ${D^\alpha }{\phi_j}$ converges to ${D^\alpha }\phi $ uniformly on $K$ for every multi-index $\al$.  Let $L= \{l \in \N: {\tx{the support of }\mcE_l}, \ stp\{\mcE_l\} \subset K \}$, then 
\[
\begin{gathered}
  \mathop {\lim }\limits_{j \to \infty } {\left\| {{D^\alpha }\phi  - {D^\alpha }{\phi _j}} \right\|_{KS}} = \mathop {\lim }\limits_{j \to \infty } \mathop {\sup }\limits_{l \in L} \left| {\int_{{\mathbb{R}^n}} {\left[ {{D^\alpha }\phi \left( x \right) - {D^\alpha }{\phi _j}\left( x \right)} \right]{{\mathcal{E}}_l}\left( x \right)d{\lambda _n}\left( x \right)} } \right| \hfill \\
   \leqslant  {\left[ {\frac{1}{{2\sqrt n }}} \right]^n} \mathop {\lim }\limits_{j \to \infty } \mathop {\sup }\limits_{x \in K} \left| {{D^\alpha }\phi \left( x \right) - {D^\alpha }{\phi _j}\left( x \right)} \right| =0. \hfill \\ 
\end{gathered} 
\]
It follows that $\mcD[\mathbb{R}^n] \subset {KS}^p [\mathbb{R}^n] $ as a continuous embedding, for $1 \le p \le \iy$.  Thus, by the Hahn-Banach theorem, we see that the Schwartz distributions, $\mcD'[\mathbb{R}^n] \subset [{KS}^p (\mathbb{R}^n)]' $, for $1 \le p \le \iy$.
\end{proof}
We close this section with the following result that will be important later.   
\begin{lem}  Let the Fourier transform, $\mathfrak{F}$ and the convolution operator, $\mathfrak{C}$ be defined on ${{L}}^1 [{\mathbb{R}}^n ]$.  Then each has a bounded extension to the  linear operators on $K{S}^2[{\mathbf{R}}^n ] $. 
\end{lem}
\begin{proof} From Theorem \rf{be}, every bounded linear operator on ${{L}}^1 [{\mathbb{R}}^n ]$ extends to a bounded linear operator on ${KS}^2 [{\mathbb{R}}^n ] $.  The theorem applies to $\mathfrak{F}$ and $\mathfrak{C}$. 
\end{proof}
\section{The Jones $SD^p$ Spaces}
For our second class of spaces, we begin with the construction of a special class of functions in $\C_c^{\iy}[\R^n]$ (see Jones, \cite{J} page 249). 
\subsection{The remarkable Jones functions}  
\begin{Def} For $x \in \R, \ 0 \le y < \iy$ and $1<a< \iy$, define the Jones functions  $g(x, y), \ h(x)$ by:
\[
  g(x,y) = \exp \left\{ { - y^a e^{iax} } \right\},
\]
\[ 
 h(x)=\left\{
\begin{array}{ll}
\displaystyle{\int_0^\infty g(x,y) dy}, & x\in  (-\frac{\pi}{2a},\frac{\pi}{2a}) \\
~&~\\
0, & \mbox{otherwise.}
\end{array}
\right.\]
\end{Def}
The following properties of $g$ are easy to check:
\begin{enumerate}
\item 
\[
\frac{{\partial g(x,y)}}
{{\partial x}} =  - iay^a e^{iax} g(x,y),
\]
\item
\[
\frac{{\partial g(x,y)}}
{{\partial y}} =  - ay^{a - 1} e^{iax} g(x,y),
\]
so that
\item
\[
iy\frac{{\partial g(x,y)}}
{{\partial y}} = \frac{{\partial g(x,y)}}
{{\partial x}}. \; \qquad\qquad \;
\]
\end{enumerate}
It is also easy to see that $h(x) \in {{L}}^1[- \tf{\pi}{2a}, \tf{\pi}{2a}]$ and,
\beqn
\frac{{dh(x)}}
{{dx}} = \int_0^\infty  {\frac{{\partial g(x,y)}}
{{\partial x}}dy}  = \int_0^\infty  {iy\frac{{\partial g(x,y)}}
{{\partial y}}dy}.
\eeqn
Integration by parts in the last expression in (1.1) shows that $h'(x) =  - ih(x)$, so that $
h(x) = h(0)e^{ - ix}$ for $x \in (- \tf{\pi}{2a}, \tf{\pi}{2a})$.  Since $
h(0) = \int_0^\infty  {\exp \{  - y^a \} dy}$, an additional integration by parts shows that $h(0)= \G(\tf{1}{a} +1)$. For each $k \in \N$ let $a=a_k= \pi 2^{k-1},\; h(x)=h_k(x), \ x \in (- \tf{1}{2^k}, \tf{1}{2^k})$ and set $\e_k=\tf{1}{2^{k+1}}$. 

  Let $\bQ$ be the set of rational numbers in $\R$ and for each $x^i \in \bQ$,  define
\beqa
f_k^{i} (x) = f_k(x-x^i)=\left\{ {\begin{array}{*{20}c}
   {c_k \exp \left\{ {\frac{{\varepsilon _k^2 }}
{{\left| {x - x^i} \right|^2  - \varepsilon _k^2 }}} \right\},\quad\quad \left| {x - x^i } \right| < \varepsilon _k ,}  \\
   {0, \quad \quad \quad \quad \quad \quad \quad \quad \quad \quad \left| {x - x^i } \right| \geqslant \varepsilon _k ,}  \\

 \end{array} } \right.
\eeqa
where $c_k$ is the standard normalizing constant.  It is clear that the support of $f_k^i$ is 
\[
{\rm{spt}}(f_k^{i}) \subset [-\e_k, \e_k]=[-\tf{1}{2^{k+1}}, \tf{1}{2^{k+1}}]=I_k^i.
\]

If we set $\chi _k^i (x) =  (f_k^i  * h_k)(x)$, its support is ${\rm{spt}}(\chi _k^i ) \subset [-\tf{1}{2^{k+1}}, \tf{1}{2^{k+1}}]$.  For $x \in {\rm{spt}}(\chi _k^i )$, we can also write $\chi _k^i (x)=\chi _k(x-x^i)$ as:
\[
\begin{gathered}
\chi _k^i (x) \hfill \\
 \quad \quad=  \int_{ I_k^i }  {{f_k}\left[ {\left( {x - {x^i}} \right) - z} \right]{h_k}(z)dz}  \hfill \\
 \quad \quad  = \int_{ I_k^i }  {{h_k}\left[ {\left( {x - {x^i}} \right) - z} \right]{f_k}(z)dz}  \hfill \\
 \quad \quad = {e^{-i\left( {x - {x^i}} \right)}}\int_{ I_k^i }  {{e^{iz}}{f_k}(z)dz} . \hfill \\ 
\end{gathered} 
\]
Thus, if $\alpha_{k,i}  = \int_{I_k^i}  {e^{iz} f_k^i (z)dz} $, we can now define:
\[
\xi _k^i (x) =\al_{k,i}^{-1} \chi_k^i (-x)= \left\{ {\begin{array}{*{20}c}
   {\f{1}{n} {e^{i(x-x^i)}}, \; \;x \in {I_k^i} } \; \\
   {0,  \quad \quad\quad \quad x \notin {I_k^i} },  \\

 \end{array} } \right.
\]
so that $\lt|\xi _k^i (x)\rt| < \tf{1}{n}$.
\subsection{The Construction}
To construct our space on $\R^n$, let $\bQ^n$ be the set of all vectors ${\mathbf{x}}$ in $ {\mathbb{R}}^n$, such that for each $j, \; 1 \le j \le n$, the component $x_j$ is rational.  Since this is a countable dense set in ${\mathbb{R}}^n $, we can arrange it as $\mathbb{Q}^n  = \left\{ {{\mathbf{x}}^1, {\mathbf{x}}^2, {\mathbf{x}}^3,  \cdots } \right\}$.  For each $k$ and $i$, let ${\mathbf{B}}_k ({\mathbf{x}}^i ) $ be the closed cube centered at ${\mathbf{x}}^i$ with  edge $e_k=\tfrac{1 }{2^{k}\sqrt{n}}$.   

We choose the natural order which maps $\mathbb{N} \times \mathbb{N}$ bijectively to $\mathbb{N}$:
\[
\{(1,1), \ (2,1), \ (1,2), \ (1,3), \  (2,2), \  (3,1), \ (3,2), \ (2,3), \  \ldots \}
\]
 and let $\left\{ {{\mathbf{B}}_{m} ,\;m \in \mathbb{N}}\right\}$ be the set of closed cubes 
${\mathbf{B}}_k ({\mathbf{x}}^i )$ with $(k,i) \in \mathbb{N} \times \mathbb{N}$ and ${\bf{x}}^i \in\mathbb{Q}^n $. For each ${\bf{x}} \in {\bf{B}}_m, \; {\bf{x}}=(x_1, x_2, \dots, x_n)$, we define  $\mcE_m ({\mathbf{x}})$ by :
\[
  { {\mathcal{E}}_m}({\mathbf{x}}) = \left( {\xi _k^i({x_1}),\xi _k^i({x_2}) \ldots \xi _k^i({x_n})} \right). 
\]
It is easy to show that, for $m=(k,i)$,
\[
\begin{gathered}
  \left| {{ {\mathcal{E}}_m}({\mathbf{x}})} \right| < {1},\;\;{\mathbf{x}} \in \prod\nolimits_{j = 1}^n {I_k^i}, \hfill \\
{ {\mathcal{E}}_m}({\mathbf{x}}) = 0,\;{\text{ }}{\mathbf{x}} \notin \prod\nolimits_{j = 1}^n {I_k^i} . \hfill \\
\end{gathered} 
\]
It is also easy to see that ${\mathcal{E}}_m ({\mathbf{x}})$ is in ${{{L}}^p [{\mathbb{R}}^n ]^n }={{\bf{L}}^p [{\mathbb{R}}^n ] }$ for $1 \le p \le\infty$.  Define $F_{m} (\; \cdot \;)$ on $
{{\bf{L}}^p [{\mathbb{R}}^n ] }$ by
 \beqa
F_{m} (f) = \int_{{\mathbb{R}}^n } {{\mathcal{E}}_{m} ({\mathbf{x}}) \cdot f({\mathbf{x}})d\la_n({\mathbf{x}})}. 
\eeqa
It is clear that $F_{m} (\; \cdot \;)$ is a bounded linear functional on ${{\bf{L}}^p [{\mathbb{R}}^n ]} $ for each ${m}$ with $ \left\| {F_{m} } \right\|   \le 1$.  Furthermore,  if $F_m (f) = 0$ for all ${m}$, $f = 0$ so that $\left\{ {F_{m} } \right\}$ is a fundamental sequence of functionals  on ${{\bf{L}}^p [{\mathbb{R}}^n ]} $ for $1 \le p \le \infty$.

Set ${t_{m}}= \tfrac{1}{{2^m }} $ so that ${\sum\nolimits_{m = 1}^\infty  {t_m}}=1$  and  define a  inner product $\left( {\; \cdot \;} \right) $ on ${{\bf{L}}^1 [{\mathbb{R}}^n ]} $ by
\beqa
 \left( {f,g} \right) = \sum\nolimits_{m = 1}^\infty  {t_m } \left[ {\int_{\mathbb{R}^n } {{\mathcal{E}}_m ({\mathbf{x}})\cdot f({\mathbf{x}})d\la_n({\mathbf{x}})} } \right] \overline{\left[ {\int_{\mathbb{R}^n } {{\mathcal{E}}_m ({\mathbf{y}}) \cdot g({\mathbf{y}})d\la_n({\mathbf{y}})} } \right]}.    
\eeqa
 The completion of ${{\bf{L}}^1 [{\mathbb{R}}^n ]} $ with the above inner product is a Hilbert space, which we denote as ${{SD}}^2 [{\mathbb{R}}^n ] $. 
 For our next theorem, we recall that $\mfM[\R^n]$ is the space of all (finite) complex measures on $\mfB[\R^n]$ that are absolutely continuous with respect to Lebesgue measure $\la_n$ and that, a sequence of measures $(\mu_j) \subset \mfM[\R^n]$, converges weakly to a measure $\mu \in \mfM[\R^n]$ if and only if, for every bounded continuous function $h$ on $\R^n, \; h \in \C(\R^n)$,    
$
\mathop {\lim }\nolimits_{j \to \infty } \left| {\int_{{\mathbb{R}^n}} {h({\mathbf{x}})d{\mu _j}({\mathbf{x}})}  - \int_{{\mathbb{R}^n}} {h({\mathbf{x}})d{\mu}({\mathbf{x}})} } \right| = 0.
$   
The proofs of the following theorem is the same as for the ${KS}^p [{\mathbb{R}}^n ] $ spaces, so we omit them.
\begin{thm} For each $p,\;1 \leqslant p \leqslant \infty$, we have:
\begin{enumerate}
\item The space ${SD}^2 [{\R}^n ] \supset {\bf{L}}^p [\R^n ]$ as a continuous, dense and \underline{compact}  embedding.
\item The space ${{SD}}^2 [{\R}^n ] \supset \mathfrak{M}[{\R}^n ]$, the space of finitely additive measures on $\R^n$, as a continuous dense and \underline{compact}  embedding.
\end{enumerate}
\end{thm}
\begin{Def}We call ${{SD}}^2 [{\mathbb{R}}^n ] $ the Jones strong distribution Hilbert space on $\R^n$.
\end{Def}
In order to justify our definition, let $\al$ be a multi-index of nonnegative integers, $\al = (\al_1, \ \al_2, \ \cdots \ \al_k)$, with $\left| \alpha  \right| = \sum\nolimits_{j = 1}^k {\alpha _j } $.   If $D$ denotes the standard partial differential operator, let 
$D^{\al}=D^{\al_1}D^{\al_2} \cdots D^{\al_k}$.
\begin{thm} Let $\mcD[\R^n]$  be $\C_c^\iy[\R^n]$ equipped with the standard locally convex topology (test functions).
\begin{enumerate}
\item If $\phi_j \to \phi$ in $\mcD[\R^n]$, then $\phi_j \to \phi$ in the norm topology of  $ SD^2[\R^n]$, so that $\mcD[\R^n] \subset {{SD}}^2[{\mathbb{R}}^n ]$ as a continuous dense embedding.
\item If $T \in \mcD'[\R^n]$, then $T \in {SD^2[\R^n]}'$, so that $\mcD'[\R^n] \subset {{SD}^2[{\mathbb{R}}^n ]}'$ as a continuous dense embedding.
\item For any $f, \ g   \in SD^2[\R^n]$ and any multi-index $\al, \ {\left( {{D^\alpha }f,g} \right)_{SD}}={(-i)^\alpha }{\left( {f,g}\right)_{SD}}$.
\end{enumerate}
\end{thm}
\begin{proof} The proofs of (1) and (2) are easy.   To prove (3), we use the fact that each ${\mathcal{E}}_m \in \C_c^{\iy}[{\mathbb{R}^n }]$.  Thus, for any $f \in SD^2[\R^n]$ we have: 
\[ 
{\int_{\mathbb{R}^n } {{\mathcal{E}}_m ({\mathbf{x}}) \cdot D^{\al}{{f}}({\mathbf{x}})d\la_n({\mathbf{x}})} }= (-1)^{\lt|\al \rt|}{\int_{\mathbb{R}^n } D^{\al} {{\mathcal{E}}_m ({\mathbf{x}}) \cdot{{f}{}}({\mathbf{x}})d\la_n({\mathbf{x}})} }.
\]
An easy calculation shows that: 
\[ 
(-1)^{\lt|\al \rt|}{\int_{\mathbb{R}^n } {D^{\al}{\mathcal{E}}_m ({\mathbf{x}}) \cdot {{f}}({\mathbf{x}})d\la_n({\mathbf{x}})} }= (-i)^{\lt|\al \rt|}{\int_{\mathbb{R}^n } {{\mathcal{E}}_m ({\mathbf{x}}) \cdot{{f}{}}({\mathbf{x}})d\la_n({\mathbf{x}})} }.
\]
It now follows that, for any ${\bf g}  \in {{SD}}^2[{\mathbb{R}}^n ], \; (D^{\al}{{f}}, {\bf g})_{{{SD}}^2}= (-i)^{\lt|\al \rt|}({{f}{}},{\bf g})_{{{SD}}^2}$.
\end{proof} 
\subsection{Functions of Bounded Variation}
The objective of this section is to show that every  HK-integrable function is in  ${{SD}}^2 [{\mathbb{R}}^n ] $.  To do this, we need to discuss a certain class of  functions of bounded variation.  For functions defined on $\R$, the definition of bounded variation is unique.  However, for functions on $\R^n, \; n \ge 2$, there are a number of distinct definitions.

The functions of  bounded variation in the sense of Cesari are well known to analysts working in partial differential equations  and geometric measure theory (see Leoni \cite{GL}).  
\begin{Def}A function $f \in L^1[\R^n]$ is said to be of bounded variation in the sense of Cesari or $f \in BV_c[\R^n]$, if $f \in L^1[\R^n]$ and each $i, \; 1 \le i \le n$, there exists a signed Radon measure $\mu_i$, such that  
\[
\int_{\mathbb{R}^n } {f({\mathbf{x}})\frac{{\partial \phi ({\mathbf{x}})}}
{{\partial x_i }}d\lambda _n ({\mathbf{x}})}  =  - \int_{\mathbb{R}^n } {\phi ({\mathbf{x}})d\mu _i ({\mathbf{x}})} ,
\]
for all $\phi \in \C_0^\iy[\R^n]$.
\end{Def} 

The functions of  bounded variation in the sense of Vitali \cite{TY1}, are well known to applied mathematicians and engineers with interest in error estimates associated with research in control theory, financial derivatives, high speed networks, robotics and in the calculation of certain integrals.   (See, for example \cite{KAA}, \cite{{NI}}, \cite{PT} or \cite{PTR} and references therein.)
For the general definition, see Yeong (\cite{TY1}, p. 175).  We present a definition that is sufficient for continuously  differentiable functions.
\begin{Def} A function $f$ with continuous partials is said to be of bounded variation in the sense of Vitali or $f \in BV_v[\R^n]$ if for all intervals $(a_i,b_i), \, 1 \le i \le n$,
\[
V(f)=\int_{a_1 }^{b_1 } { \cdots \int_{a_n }^{b_n } {\left| {\frac{{\partial ^n f({\mathbf{x}})}}
{{\partial x_1 \partial x_2  \cdots \partial x_n }}} \right|d\lambda _n ({\mathbf{x}})} }  < \infty. 
\]
\end{Def}
\begin{Def}We define $BV_{v,0}[\R^n]$ by: 
\[
BV_{v,0}[\R^n]= \{f({\bf x}) \in BV_v[\R^n]: f({\bf x}) \to 0, \; {\rm as} \; x_i \to -\iy \},
\]
where $x_i$ is any component of ${\bf x}$. 
\end{Def}
The following two theorems may be found in \cite{TY1}. (See p. 184 and 187, where the first is used to prove the second.)  Recall that, if $[a_i, b_i] \subset {\bar \R}=[-\iy, \iy]$, we define $[{\bf a}, {\bf b}] \in {\bar \R}^n$ by   $[{\bf a}, {\bf b}]=\prod_{k=1}^n{[a_i,b_i]}$.  (The notation $(RS)$ means Riemann-Stieltjes.)
\begin{thm}Let $f$ be HK-integrable on $\left[ {{\mathbf{a}},{\mathbf{b}}} \right]$ and let $g \in BV_{v,0}[\R^n]$, then $fg$ is 
HK-integrable and 
\[
(HK)\int_{[{\bf a}, {\bf b}] }{f({\bf x})g({\bf x})d\la_n({\bf x})}=(RS)\int_{[{\bf a}, {\bf b}] }{\lt\{(HK)\int_{[{\bf a}, {\bf x}] }f({\bf y})d\la_n({\bf y})\rt\}dg({\bf x})}
\] 
\end{thm}
\begin{thm}Let $f$ be HK-integrable on $\left[ {{\mathbf{a}},{\mathbf{b}}} \right]$ and let $g \in BV_{v,0}[\R^n]$, then $fg$ is 
HK-integrable and 
\[
\lt|(HK)\int_{[{\bf a}, {\bf b}] }{f({\bf x})g({\bf x})d\la_n({\bf x})}\rt|
\le \lt\|f\rt\|_{A_n}V_{[{\bf a}, {\bf b}] }(g)
\]
\end{thm}
\begin{lem}\lb{8}The space $HK[{\mathbb{R}}^n ]$, of all HK-integrable functions is contained in ${{SD}}^2 [{\mathbb{R}}^n ] $.
\end{lem}
\begin{proof}
Since each $\mcE_m({\bf x})$ is  continuous and differentiable, $\mcE_m({\bf x}) \in BV_{v,0}[\R^n]$, so that for  $f \in HK[{\mathbb{R}}^n ] $,
\[
\begin{gathered}
\left\| f \right\|_{{\bf{SD}}^2}^2  = \sum\nolimits_{m = 1}^\infty  {t_m } \left| {\int_{\mathbb{R}^n } {{\mathcal{E}}_m ({\mathbf{x}})\cdot f({\mathbf{x}})d{\mathbf{x}}} } \right|^2  \leqslant \sup _m \left| {\int_{\mathbb{R}^n } {{\mathcal{E}}_m ({\mathbf{x}})\cdot f({\mathbf{x}})d{\mathbf{x}}} } \right|^2  \hfill \\
\leqslant \left\| f \right\|_{A_n}^2 [\sup _m V(\mcE_m)]^2 <\iy. \hfill \\
\end{gathered} 
\]
It follows that $f \in {{SD}}^2 [{\mathbb{R}}^n ]$. 
\end{proof}  
\subsection{{{The General Case,} ${{SD}}^p,\; 1 \le p \le \iy$ }} 
To construct ${SD}^p [{\mathbb{R}}^n ]$ for all $p$ and  for ${\bf f} \in {\bf{L}}^p$, define:
\[
\left\| {\bf f} \right\|_{{{SD}}^p }  = \left\{ {\begin{array}{*{20}c}
   {\left\{ {\sum\limits_{m = 1}^\infty  {t_m \left| {\int_{\mathbb{R}^n } { {\mathcal{E}}_m ({\mathbf{x}}) \cdot {\bf f}({\mathbf{x}})d\la_n({\mathbf{x}})} } \right|} ^p } \right\}^{1/p} ,1 \leqslant p < \infty},  \\
 \mathop {\sup }\limits_{m \geqslant 1} \left| {\int_{\mathbb{R}^n } { {\mathcal{E}}_m ({\mathbf{x}}) \cdot {\bf f}({\mathbf{x}})d\la_n({\mathbf{x}})} } \right|, \;\; \quad \quad \quad \quad \quad p = \infty .  \\

 \end{array} } \right.
\] 
It is easy to see that $\left\| \cdot \right\|_{{{SD}}^p }$ defines a norm on  ${\bf{L}}^p$.  If ${{{SD}}^p }$ is the completion of ${\bf{L}}^p$ with respect to this norm, we have:
\begin{thm} For each $q,\;1 \leqslant q \leqslant \infty,$
 ${SD}^p [{\mathbb{R}}^n ] \supset {\bf{L}}^q [{\mathbb{R}}^n ]$ as a continuous dense  embedding.
\end{thm}
\begin{thm} For ${SD}^p$, $1\leq p \leq \infty$, we have: 
\begin{enumerate}
\item If $p^{ - 1}  + q^{ - 1}  = 1$, then the dual space of ${SD}^p[{\mathbb{R}}^n ]$ is ${SD}^q[{\mathbb{R}}^n ]$. \item  For all $f \in {SD}^p[{\mathbb{R}}^n ], g \in {SD}^q[{\mathbb{R}}^n ]$ and all multi-index $\al$, $\left\langle {{D^\alpha }f,g} \right\rangle  = {\left( { - i} \right)^{\left| \alpha  \right|}}\left\langle {f,g} \right\rangle $.
\item The test function space ${\mcD}[{\mathbb{R}}^n ]$ is contain in ${SD}^{p}[{\mathbb{R}}^n ]$ as a continuous dense embedding.
\item If $K$ is a weakly compact subset of ${\bf{L}}^p[{\mathbb{R}}^n ]$, it is a strongly compact subset of ${SD}^{p}[{\mathbb{R}}^n ]$.
\item The space $ {SD}^{\infty}[{\mathbb{R}}^n ] \subset {SD}^p[{\mathbb{R}}^n ]$.
\end{enumerate}
\end{thm}
\begin{rem}The fact that the families ${{KS}}^p[{\mathbb{R}}^n ]$ and ${{SD}}^p[{\mathbb{R}}^n ]$ are  separable and contain ${L}^\iy$ as a continuous dense compact embedding, makes it clear that the relationship between analysis and topology is not as straight forward as one would expect from past history.  Thus, from an analysis point of view they are big, but from a topological point of view they are relatively small (separable).
\end{rem}
\section{Zachary Spaces}
In this section, we discuss one new space and two other families of spaces that naturally flow from the existence of a Banach space structure for functions with a bounded integral. 
\subsection{Functions of Bounded and Weak Bounded Mean Oscillation}
In this section, we first define the functions of bounded mean oscillation ($BMO$) based on the sharp maximal function ($M^\#$).  In the following section, we define a weak maximal function ($M^w$) and use it construct the space of functions $BMO^w$, which extends $BMO$ to include the functions with a bounded integral. 
\subsubsection{Sharp maximal function and $BMO$}
\begin{Def}  Let $f \in L_{\rm loc}^1[\R^n]$ and let $Q$ be a cube in $\R^n$.
\begin{enumerate}
\item We define the average of $f$ over $Q$ by 
\[
\mathop {Avg}\limits_Q f = \frac{1}
{{\lambda _n \left[ Q \right]}}\int_Q {f({\mathbf{y}})d\lambda _n ({\mathbf{y}})} .
\]
\item We defined the sharp maximal function $M^\#(f)({\bf x})$, by
\[
M^\#  (f)({\mathbf{x}}) = \mathop {\sup }\limits_Q \frac{1}
{{\lambda _n \left[ Q \right]}}\int_Q {\left| {f({\mathbf{y}}) - \mathop {Avg}\limits_Q f} \right|d\lambda _n ({\mathbf{y}})} ,
\]
where the supremum is over all cubes containing $\bf{x}$. 
\item If $M^\#  (f)({\mathbf{x}}) \in L^\iy[\R^n]$, we say that $f$ is of bounded mean oscillation.  More precisely, the space of functions of bounded mean oscillation are defined by: 
\[
BMO[{\mathbb{R}^n}] = \left\{ {f \in L_{loc}^1[{\mathbb{R}^n}]:\;{M^\# }(f) \in {L^\infty }[{\mathbb{R}^n}]} \right\}
\]
and 
\[{\left\| f \right\|_{BMO}} = {\left\| {{M^\# }(f)} \right\|_{{L^\infty }}}.\]
\end{enumerate}
\end{Def}
We may also obtain an equivalent definition of $BMO[\R^n]$ using balls, but for our purposes, cubes are natural (see Grafakos \cite{GRA} p. 546). We note that $BMO[\R^n]$ is not a Banach space and is not separable. In the next section, we construct the space $BMO^w[\R^n]$, that contains $BMO[\R^n]$ as a continuous dense embedding.
\subsubsection{Weak maximal function and $BMO^w$}
\begin{Def}  Let $f \in L_{\rm loc}^1[\R^n]$ and let $Q$ be a cube in $\R^n$.
\begin{enumerate}
\item We define $f_{aQ}$ over $Q$ by 
\[
{f_{aQ}} = \left| {\frac{1}{{{\lambda _n}\left( Q \right)}}\int_Q {f({\mathbf{y}})} d{\lambda _n}\left( {\mathbf{y}} \right)} \right| = \left| {\mathop {Avg}\limits_Q f} \right| .
\]
\item We defined the weak maximal function $M^w(f)({\bf x})$, by
\[
{M^w}\left( f \right)\left( {\mathbf{x}} \right) = \mathop {\sup }\limits_Q \left| {\frac{1}{{{\lambda _n}\left( Q \right)}}\int_Q {\left[ {f({\mathbf{y}}) - {f_{aQ}}} \right]} d{\lambda _n}\left( {\mathbf{y}} \right)} \right| ,
\]
where the supremum is over all cubes containing $\bf{x}$. 
\item If $M^w(f)({\mathbf{x}}) \in L^\iy[\R^n]$, we say that $f$ is of weak bounded mean oscillation.  We define $BM$ by 
\[
BM=\left\{ {f\left( {\mathbf{x}} \right) \in {L_{{\text{loc}}}^1}\left[ {{\mathbb{R}^n}} \right]:  {M^w}\left( f \right)\left( {\mathbf{x}} \right) \in {L^\infty }\left[ {{\mathbb{R}^n}} \right]} \right\}
\] 
and define a seminorm on $BM$ by
\[
{\left\| f \right\|_{BM{O^w}}} = {\left\| {{M^w}\left( f \right)} \right\|_\infty }.
\]
\end{enumerate}
\end{Def}
\begin{Def}  We define $BMO^w[\R^n]$ to be the completion of $BM$ in the seminorm ${\left\| \cdot \right\|_{BMO^w}}$.
\end{Def}
\begin{rem}\lb{o}  If  ${\left\| f \right\|_{BMO^w}} =0$, then  for every cube $Q$ containing $\bf{x}$,  
\[
\frac{1}{{{\lambda _n}\left( Q \right)}}\int_Q {f({\mathbf{y}})} d{\lambda _n}\left( {\mathbf{y}} \right) - {f_{aQ}} = 0 \Rightarrow \frac{1}{{{\lambda _n}\left( Q \right)}}\int_Q {f({\mathbf{y}})} d{\lambda _n}\left( {\mathbf{y}} \right) = \left| {\frac{1}{{{\lambda _n}\left( Q \right)}}\int_Q {f({\mathbf{y}})} d{\lambda _n}\left( {\mathbf{y}} \right)} \right|.
\]
Since $a = \left| a \right|$ if and only if $a \ge 0$, we see that $f({\mathbf{x}})$ is  a nonnegative constant (a.e).
It follows that  $BMO^w[\R^n]$ is not a Banach space, but becomes one if we identify terms that differ by a nonnegative constant. 
\end{rem}
\begin{thm}The space $BMO[\R^n] \subset BMO^w[\R^n]$ as a continuous dense embedding.
\end{thm}
\begin{proof}
It is easy to see that $BMO[\R^n]$ is a dense subset.  To prove that its a continuous embedding, we note that
\[
\begin{gathered}
  {M^w}\left( f \right)\left( {\mathbf{x}} \right) =  \hfill \\
  \mathop {\sup }\limits_Q \frac{1}{{{\lambda _n}\left( Q \right)}}\left| {\int_Q {\left[ {f({\mathbf{y}}) - \left| {\mathop {Avg}\limits_Q f} \right|} \right]} d{\lambda _n}\left( {\mathbf{y}} \right)} \right| \leqslant \mathop {\sup }\limits_Q \frac{1}{{{\lambda _n}\left( Q \right)}}\int_Q {\left| {f({\mathbf{y}}) - \left| {\mathop {Avg}\limits_Q f} \right|} \right|} d{\lambda _n}\left( {\mathbf{y}} \right) \hfill \\
   \leqslant \mathop {\sup }\limits_Q \frac{1}{{{\lambda _n}\left( Q \right)}}\int_Q {\left| {f({\mathbf{y}}) - \mathop {Avg}\limits_Q f} \right|} d{\lambda _n}\left( {\mathbf{y}} \right) = {M^\# }\left( f \right)\left( {\mathbf{x}} \right). \hfill \\ 
\end{gathered} 
\]
\end{proof}
\begin{cor} The space $BMO^w[\R^n] \subset KS^\iy[\R^n]$ as a continuous dense embedding.
\end{cor}
\begin{proof}The proof is easy since $L_{\rm loc}^1[\R^n] \cup L^\iy[\R^n] \subset KS^\iy[\R^n]$, with $L^\iy[\R^n]$ dense and 
\[
\begin{gathered}
  {\left\| {{M^w}(f)} \right\|_{K{S^\infty }}} \hfill \\
   = \mathop {\sup }\limits_k \left| {\int_{{\mathbb{R}^n}} {{\mcE_k}({\mathbf{x}}){M^w}(f)({\mathbf{x}})d{\lambda _n}({\mathbf{x}})} } \right| \leqslant {\left[ {\frac{1}{{2\sqrt n }}} \right]^n}{\left\| {{M^w}(f)} \right\|_{{L^\infty }}} = {\left[ {\frac{1}{{2\sqrt n }}} \right]^n}{\left\| f \right\|_{{M^w}}}. \hfill \\ 
\end{gathered} 
\] 
\end{proof}
\subsection{Zachary Functions of Bounded Mean Oscillation}
We now construct another class of functions.  Let the family of cubes $\{ {Q_k } \}$ centered at each rational point in $\R^n$ be the ones  generated by the indicator functions   $\{ {\mathcal{E}}_k ({\mathbf{x}}) \}$, for ${{KS}}^2 [{\mathbb{R}}^n ]$.  Let $f \in L_{\rm loc}^1[\R^n]$ and as before, we define $f_{ak}$ by
\[
f_{ak}  = \lt| \frac{1}
{{\lambda _n \left[ {Q_k } \right]}}\int_{Q_k } {f({\mathbf{y}})d\lambda _n ({\mathbf{y}})}\rt|  =\lt| \frac{1}
{{\lambda _n \left[ {Q_k } \right]}}\int_{\mathbb{R}^n } {\mcE_k ({\mathbf{y}})f({\mathbf{y}})d\lambda _n ({\mathbf{y}})}\rt|. 
\]
\begin{Def}  If $p, \; 1 \le p <\iy$ and $t_k=2^{-k}$, we define $\left\| f \right\|_{\mcZ^p }$ by  
\[
\left\| f \right\|_{\mcZ^p }  = \left\{ {\sum\limits_{k =1 }^\infty  {t_k } \left| \frac{1}
{{\lambda _n \left[ {Q_k } \right]}}  {\int_{Q_k } {\left[ {f({\mathbf{y}}) - f_{ak} } \right]d\lambda _n ({\mathbf{y}})} } \right|^p } \right\}^{1/p} . 
\]
The set of  functions for which $\left\| f \right\|_{\mcZ^p } <\iy $ is called the Zachary functions of bounded mean oscillation and order $p, \; 1 \le p <\iy$.  If $p= \iy$, we say that   $f \in \mcZ^{\iy}[\R^n]$ if
\[
\left\| f \right\|_{\mcZ^\iy }  =\mathop {\sup }\limits_{k} \left| \frac{1}
{{\lambda _n \left[ {Q_k } \right]}}  {\int_{Q_k } {\left[ {f({\mathbf{y}}) - f_{ak} } \right]d\lambda _n ({\mathbf{y}})} } \right| <\iy.
\]
\end{Def}
The following theorem shows how the Zachary spaces are related to the space of functions of Bounded mean oscillation $BMO[\R^n]$.  (We omit proofs.)
\begin{thm}If $\mcZ^p[\R^n]$ is the class of Zachary functions of bounded mean oscillation and order $p, \; 1\le p \le \iy$, then $\mcZ^p[\R^n]$ is a linear space and
\begin{enumerate}
\item  $\left\| {\lambda f} \right\|_{\mcZ^p }  = \left| \lambda  \right|\left\| f \right\|_{\mcZ^p } $.
\item $\left\| {f + g} \right\|_{\mcZ^p }  \leqslant \left\| f \right\|_{\mcZ^p }  + \left\| g \right\|_{\mcZ^p }$.
\item $\left\| f \right\|_{\mcZ^p }  = 0, \Rightarrow f \ge 0$ is a nonnegative constant (a.e).
\item The space $\mcZ^{\iy}[\R^n] \subset {\mcZ^p[\R^n] }, \; 1 \le p < \iy$, as a dense continuous embedding.
\item $BM{O^w}\left[ {{\mathbb{R}^n}} \right] \subset {{\mathcal{Z}}^\infty }\left[ {{\mathbb{R}^n}} \right]$, as a dense continuous embedding.
\end{enumerate}
\end{thm}
\begin{proof} The first four are clear.  To prove  (5), suppose that $f \in BMO^w[\R^n]$, then by definition of $\lt\| \cdot \rt\|_{BMO^w}$, the supremum is over the set of all cubes in $\R^n$. Since this set is much larger then the countable number used to define $\lt\| \cdot \rt\|_{\mcZ^\iy}$.  It follows that $BM{O^w}\left[ {{\mathbb{R}^n}} \right] \subset {{\mathcal{Z}}^\infty }\left[ {{\mathbb{R}^n}} \right]$ as a dense continuous embedding.
\end{proof} 
We now consider the Carleson measure characterization of $BMO[\R^n]$ which will prove useful in construction another class of Zachary spaces that are Banach spaces (see Koch and Tataru \cite{KT}).  If $u({\bf x}, t)$ is a solution of the heat equation: 
\[
u_t -{\De{u}}=0, \; u({\bf x}, 0) =f({\bf x}),
\]
where $f \in L_{loc}^1[\R^n]$, it can be shown that
\[
\left\| f \right\|_{BMO}  = \mathop {\sup }\limits_{{\mathbf{x}},r} \left\{ {\frac{1}
{{\lambda _n \left[ {Q({\mathbf{x}},r)} \right]}}\int_{Q({\mathbf{x}},r)} {\int_0^{r^2 } {\left| {\nabla u({\mathbf{y}},s)} \right|^2 dsd\lambda _n ({\mathbf{y}})} } } \right\}^{1/2},
\]
where the gradient is in the weak sense.
Since
\[
\begin{gathered}
  \mathop {\sup }\limits_{k,r} {\left\{ {\frac{1}{{{\lambda _n}\left[ {{Q_k}} \right]}}{{\left| {\int_{{Q_k}} {\int_0^{{r^2}} {\nabla u\left( {{\mathbf{y}},s} \right)ds} d{\lambda _n}\left( {\mathbf{y}} \right)} } \right|}^2}} \right\}^{1/2}} \hfill \\
   \leqslant \mathop {\sup }\limits_{{\mathbf{x}},r} {\left\{ {\frac{1}{{{\lambda _n}\left[ {Q\left( {{\mathbf{x}},r} \right)} \right]}}\int_{Q\left( {{\mathbf{x}},r} \right)} {\int_0^{{r^2}} {{{\left| {\nabla u\left( {{\mathbf{y}},s} \right)} \right|}^2}dsd{\lambda _n}\left( {\mathbf{y}} \right)} } } \right\}^{1/2}}, \hfill \\ 
\end{gathered} 
\]
we see that we can also define the seminorms on $BMO^w[\R^n]$  and $\mcZ^{p}[\R^n]$  by:
\[
\begin{gathered}
{\left\| f \right\|_{BM{O^w}}} = \mathop {\sup }\limits_{{\mathbf{x}},r} {\left\{ {\frac{1}{{{\lambda _n}\left[ {Q\left( {{\mathbf{x}}.r} \right)} \right]}}{{\left| {\int_{Q\left( {{\mathbf{x}}.r} \right)} {\int_0^{{r^2}} {\nabla u\left( {{\mathbf{y}},s} \right)dsd{\lambda _n}({\mathbf{y}})} } } \right|}^2}} \right\}^{1/2}}, \hfill \\
  {\left\| f \right\|_{{\mcZ^p}}} = \mathop {\sup }\limits_r {\left\{ {\sum\limits_{k = 1}^\infty  {{t_k}\frac{1}{{{\lambda _n}\left[ {{Q_k}} \right]}}{{\left| {\int_{{Q_k}} {\int_0^{{r^2}} {\nabla u({\mathbf{y}},s)ds} d{\lambda _n}({\mathbf{y}})} } \right|}^p}} } \right\}^{1/p}}. \hfill \\ 
\end{gathered} 
\]
\subsection{The Space of  Functions $BMO^{-1}[\R^n]$}
We define the class of functions $BMO^{-1}[\R^n]$, as those for which:
\[
\left\| f \right\|_{BMO^{ - 1} }  = \mathop {\sup }\limits_{{\mathbf{x}},r} \left\{ {\frac{1}
{{\lambda _n \left[ {Q({\mathbf{x}},r)} \right]}}\int_{Q({\mathbf{x}},r)} {\int_0^{r^2 } {\left| {u({\mathbf{y}},s)} \right|^2 dsd\lambda _n ({\mathbf{y}})} } } \right\}^{1/2}  < \infty .
\]
It is known that $BMO^{-1}[\R^n]$ is a Banach space in the above norm.
\begin{Def}  We say  $f \in \mcZ^{-p}[\R^n], \; 1\le p < \iy$  if
\[
\left\| f \right\|_{\mcZ^{ - p} }  =\mathop {\sup }\limits_r {\left\{ {\sum\limits_{k = 1}^\infty  {{t_k}\frac{1}{{{\lambda _n}\left[ {{Q_k}} \right]}}{{\left| {\int_{{Q_k}} {\int_0^{{r^2}} { u({\mathbf{y}},s)ds} d{\lambda _n}({\mathbf{y}})} } \right|}^p}} } \right\}^{1/p}}  < \infty.
\]
If $p= \iy$, we say that   $f \in \mcZ^{-\iy}[\R^n]$ if
\[
\left\| f \right\|_{\mcZ^{ - \infty } }  = \mathop {\sup }\limits_{k, \,  r} \frac{1}
{{\lambda _n \left[ {Q_k } \right]}}\left| {\int_{Q_k } {\int_0^{r^2 } { u({\mathbf{y}},s)dsd\lambda _n ({\mathbf{y}})} } } \right| < \infty .
\]
\end{Def}
\begin{thm}For the class of spaces $\mcZ^{-p}[\R^n]$, we have: 
\begin{enumerate}
\item  For each $p, \; 1 \le p \le \iy \;, \mcZ^{-p}[\R^n]$ is a Banach space.
\item The space $\mcZ^{-\iy}[\R^n] \subset {\mcZ^{-p}[\R^n] }, \; 1 \le p < \iy$, as a dense continuous embedding.
\item The space $BMO^{-1}[\R^n] \subset {\mcZ^{-\iy}[\R^n] }$ as a dense continuous embedding. 
\end{enumerate}
\end{thm}
\begin{proof} The first two are obvious.  To prove (3), if $f \in BMO^{-1}[\R^n]$, then 
\[
\begin{gathered}
  \left\| f \right\|_{\mcZ^{ - \iy} }  = \mathop {\sup }\limits_{k, \, r} \frac{1}
{{\lambda _n \left[ {Q_k } \right]}}\left| {\int_{Q_k } {\int_0^{r^2 } {u({\mathbf{y}},s)dsd\lambda _n ({\mathbf{y}})} } } \right|  \hfill \\
= \mathop {\sup }\limits_{k, \, r} \frac{1}
{{\lambda _n \left[ {Q_k } \right]}}\left| {\int_{Q_k } {\int_0^{r^2 } {u({\mathbf{y}},s)dsd\lambda _n ({\mathbf{y}})} } } \right|^{2/2}  \hfill \\
   \leqslant \mathop {\sup }\limits_{{\mathbf{x}},r} \frac{1}
{{\lambda _n \left[ {{Q({\mathbf{x}},r)} } \right]}}\left\{ {\int_{{Q({\mathbf{x}},r)} } {\int_0^{r^2 } {\left| {u({\mathbf{y}},s)} \right|^2 dsd\lambda _n ({\mathbf{y}})} } } \right\}^{1/2}  = \left\| f \right\|_{BMO^{ - 1} }. \hfill \\ 
\end{gathered} 
\]
\end{proof}
\begin{rem}We could also define $BMO^{ - w}[\R^n]$ by:
\[
{\left\| f \right\|_{BM{O^{-w}}}} = \mathop {\sup }\limits_{{\mathbf{x}},r} {\left\{ {\frac{1}{{{\lambda _n}\left[ {Q\left( {{\mathbf{x}}.r} \right)} \right]}}{{\left| {\int_{Q\left( {{\mathbf{x}}.r} \right)} {\int_0^{{r^2}} { u\left( {{\mathbf{y}},s} \right)dsd{\lambda _n}({\mathbf{y}})} } } \right|}^2}} \right\}^{1/2}}.
\]  
It is easy to see that $BMO^{ - w}[\R^n]$ is a Banach space and that $BMO^{ - 1}[\R^n] \subset BMO^{ - w}[\R^n] \subset \mcZ^{-\iy}[\R^n]$ as a continuous embeddings.  We \underline{conjecture} that $BMO^{ - w}[\R^n]$ and $\mcZ^{-\iy}[\R^n]$ allow us to replace solutions of the heat equation by those of the Schr{\"o}dinger equation: $i{u_t} - \Delta u = 0,\; u({\mathbf{x}},0) = f({\mathbf{x}})$.
\end{rem}
\section{Applications} 
In this section, we consider a few applications associated with the families $KS^p[\mathbb{R}^n ]$ and $SD^p[\mathbb{R}^n ], \; 1 \le p \le \iy$. In each case, we either solve an open problem or provide a substantial improvement in methods used in a given area. 
\subsection{Markov Processes} 
In the study of Markov processes, it is well-known that semigroups associated with processes whose generators have unbounded coefficients, are not necessarily strongly continuous when defined on ${\mathbb{C}}_b [\mathbb{R}^n ]$, the space of bounded continuous functions, or ${\mathbb{UBC}}[\mathbb{R}^n ]$, the bounded uniformly continuous functions.  These are the natural spaces on which to formulate the theory.   The problem is that the generator of such a semigroup does not exist in the standard sense (is not $C_0$).   As a consequence, a number of equivalent weaker definitions have been developed in the literature. A good discussion of this and related problems see Lorenzi and Bertoldi \cite{LB}. The following is one version of convergence used to a define semigroups using the generator in these cases.
\begin{Def}
A sequence of functions $\{f_k\}$ in ${\mathbb{C}}_b [\mathbb{R}^n ]$ is said to converge to $f$ in the mixed topology, written $\tau ^M$-$\lim f_k  = f$, if and only if  $\sup _{k \in \mathbb{N}} \left\| {f_k } \right\|_\infty   \leqslant M$ and $\left\| {f_k  - f} \right\|_\infty   \to 0$ uniformly on every compact subset of $\mathbb{R}^n$.   
\end{Def}
It is clear that the family of bounded continuous functions, $\C_b[\R^n] \subset KS^p[\R^n]$ as a continuous dense embedding. 
\begin{thm} If $\{f_k\}$ converges to $f$ in the mixed topology on ${\mathbb{C}}_b [\mathbb{R}^n ]$, then $\{f_n\}$ converges to $f$ in the norm topology of ${{KS}}^p[\mathbb{R}^n ]$ for each $1 \le p \le \infty$. 
\end{thm}
\begin{proof} It suffices to prove the result for $KS^\iy$. Since $\sup _{k \in \mathbb{N}} \left\| {f_k } \right\|_{KS^\iy}  \le \sup _{k \in \mathbb{N}} \left\| {f_k } \right\|_{\C_b}$,  we must prove that $\tau ^M$-$\lim f_k  = f \ \Rightarrow \ \lim _{k \to \infty } \left\| {f_k  - f} \right\|_{{K}{S}^\iy }  = 0$.  This follows from the fact that each cube used in the definition of the  ${{K}{S}^\iy[\mathbb{R}^n ]}$ norm, is a compact subset of $\mathbb{R}^n $.
\end{proof}
\begin{thm} Suppose that $\hat{T}(t)$ is a transition semigroup defined on  ${\mathbb{C}}_b [\mathbb{R}^n ]$, with weak generator $\hat{A}$.  Let $T(t)$ be the extension of $\hat{T}(t)$ to  ${{KS}}^p[\mathbb{R}^n ]$.  Then $T(t)$ is strongly continuous, and the extension $A$ of $\hat{A}$ to ${{KS}}^p[\mathbb{R}^n ]$ is the strong generator of $T(t)$. 
\end{thm}
\begin{proof} Since ${\mathbb{C}}_b [\mathbb{R}^n ] \subset {{KS}}^p[\mathbb{R}^n ]$ as a continuous embedding, for $1 \le p \le \infty$, we can apply Corollary \rf{be} to show that $\hat{T}(t)$ has a bounded extension to ${{KS}}^2[\mathbb{R}^n ]$.  It is easy to see that the extended operator $T(t)$ is a semigroup.  Since the $\tau^{M}$ topology on ${\mathbb{C}}_b [\mathbb{R}^n ]$ is stronger then the norm topology on ${{KS}}^p[\mathbb{R}^n ]$, we see that the generator $A$, of $T(t)$ is strong. 
\end{proof}
\subsection{Feynman Path Integral}
Feynman's introduction of his path integral into quantum theory created a major mathematical problem, that centered around the nonexistence of a measure for his integral.   A number of analysts, beginning with Henstock \cite{HS} have advocated the HK-integral as a perfect substitute for this problem (see also Muldowney \cite{MD}).  Since then, a number of researchers have addressed the problem.  A fairly complete list of papers and books can be found in Gill and Zachary \cite{GZ}.  The book by  Johnson and Lapidus \cite{JL}  also contains additional sources.

The space ${{L}}^2[\mathbb{R}^n ]$ is perfect for the Heisenberg and Schr{\"o}dinger formulations of quantum mechanics, but fails for the Feynman formulation.  In addition, neither the physically intuitive nor computationally efficient methods of Feynman are revealed on ${{L}}^2[\mathbb{R}^n ]$.  In this section we briefly show that ${{KS}}^2[\mathbb{R}^n ]$ is the natural Hilbert space for the Feynman formulation of quantum mechanics.  This space makes it possible to preserve all the physically intuitive and computational advantages discovered by Feynman and to  represent the Heisenberg and Schr{\"o}dinger formulations. 

We assume that the reader is familiar with the HK-integral, but  give a brief discussion in one dimension  to establish notation.   (A full discussion with proofs and some interesting examples can be found in Gill and Zachary \cite{GZ}.)   
\begin{Def} Let $[a,b] \subset \mathbb{R}$, let $\delta (t)$ map $[a,b] \to (0,\infty )$, and let ${\mathbf{P}} = \{ t_0 ,\,\tau _1 ,\,t_1 ,\,\tau _2, \,\cdots ,\,\tau _n,\,t_n \}$, where 
$a = t_0  \leqslant \tau _1  \leqslant t_1  \leqslant  \cdots  \leqslant \tau _n  \leqslant t_n  = b$.  We call $\bf{P}$ an HK-partition for $\delta$, if for $1 \leqslant i \leqslant n,\;
t_{i-1} ,t_i  \, \in (\tau _{i} \, - \delta (\tau _{i} ),\tau _{i} \, + \delta (\tau _{i} ))$.
\end{Def}
\begin{Def}  The function $f(t),\, t \in [a,b]$, is said to have a HK-integral if there is an number $F[a,b]$ such that, for each $\varepsilon  > 0$, there exists a function $\delta$ from $[a,b] \to (0,\infty )$ such that, whenever ${\mathbf{P}}$ is a HK-partition for $
\delta$, then (with $\Delta{t_{i}}=t_{i}-t_{i-1}$)
\[
\left| {\sum\nolimits_{i = 1}^n {\Delta t_i f(\tau _i ) - F[a,b]} } \right| < \varepsilon. 
\]
\end{Def}  
To understand Feynman's path integral in a natural setting, we consider the free particle in non-relativistic quantum theory in ${\mathbb{R}}^3$: 
\beqn\lb{f1}
 i\hbar \frac{{\partial \psi ({\mathbf{x}},t)}}
{{\partial t}} - \frac{\hbar^2}
{{2m}}\Delta \psi ({\mathbf{x}},t) = 0,\;\psi ({\mathbf{x}},s) = \delta ({\mathbf{x}} - {\mathbf{y}}).
\eeqn
The solution can be computed directly:
\beqa
 \psi ({\mathbf{x}},t) = K\left[ {{\mathbf{x}},\,{\text{ }}t;\;{\mathbf{y}},\,{\text{ }}s} \right] = \left[ {\frac{{2\pi i\hbar (t - s)}}
{m}} \right]^{ - 3/2} \exp \left[ {\frac{{im}}
{{2\hbar }}\frac{{\left| {{\mathbf{x}} - {\mathbf{y}}} \right|^2 }}
{{(t - s)}}} \right].
\eeqa
Feynman wrote the above solution to equation (\rf{f1}) as 
\beqn
K\left[ {{\mathbf{x}},\,t;\;{\mathbf{y}},\,s} \right] = \mathop \smallint \nolimits_{{\mathbf{x}}(s) = y}^{{\mathbf{x}}(t) = x} \mathcal{D}{\mathbf{x}}(\tau )\exp \left\{ {\tfrac{{im}}
{{2\hbar }}\mathop \smallint \nolimits_s^t {\text{ }}\mathop {\left| {\tfrac{{d{\mathbf{x}}}}
{{dt}}} \right|}\nolimits^2 d\tau } \right\},
\eeqn
where
\beqn\label{f2}
\begin{gathered}
  \mathop \smallint \nolimits_{{\mathbf{x}}(s) = y}^{{\mathbf{x}}(t) = x} \mathcal{D}{\mathbf{x}}(\tau )\exp \left\{ {\tfrac{{im}}
{{2\hbar }}\mathop \smallint \nolimits_s^t \mathop {\left| {\tfrac{{d{\mathbf{x}}}}
{{dt}}} \right|}\nolimits^2 d\tau } \right\} = : \hfill \\
  {\text{        }}\mathop {\lim }\limits_{N \to \infty } \left[ {\tfrac{m}
{{2\pi i\hbar \varepsilon (N)}}} \right]^{3N/2} \mathop \smallint \nolimits_{{\mathbb{R}}^3 } \mathop \prod \limits_{j = 1}^N d{\mathbf{x}}_j \exp \left\{ {\tfrac{i}
{\hbar }\mathop \sum \limits_{j = 1}^N \left[ {\tfrac{m}
{{2\varepsilon (N)}}\mathop {\left( {{\mathbf{x}}_j  - {\mathbf{x}}_{j - 1} } \right)}\nolimits^2 } \right]} \right\}, \hfill \\ 
\end{gathered} 
\eeqn
with $
\varepsilon (N) = {{\left( {t - s} \right)} \mathord{\left/
 {\vphantom {{\left( {t - s} \right)} N}} \right.
 \kern-\nulldelimiterspace} N} $.  
Equation (\rf{f2}) is an attempt to define an integral over the space of all continuous paths of the exponential of an integral of the classical Lagrangian on configuration space.   This approach has led to a new approach for quantizing physical systems,  called the path integral method.    

Since $L^2[\R^3]$ is the standard state space for quantum physics, from a strictly mathematical point of view equation (\rf{f2}) has two major problems:  
\begin{enumerate}
\item
 The kernel $K\left[ {{\mathbf{x}},\,{\text{ }}t;\;{\mathbf{y}},\,{\text{ }}s} \right] $ and $\delta ({\mathbf{x}})$ are not in $L^2 [{\mathbb{R}}^3 ]$.
 \item 
 The kernel $K\left[ {{\mathbf{x}},\,{\text{ }}t;\;{\mathbf{y}},\,{\text{ }}s} \right] $ cannot be used to define a measure.
\end{enumerate}
Since ${K}{S}^2 [{\mathbb{R}}^3 ] $ contains the space of measures $\mfM[\R^3]$, it follows that all the approximating sequences for the Dirac measure converge strongly to it in the ${K}{S}^2 [{\mathbb{R}}^n ] $ topology.  (For example, $
\left[ {{{\sin (\lambda  \cdot {\mathbf{x}})} \mathord{\left/
 {\vphantom {{\sin (\lambda  \cdot {\mathbf{x}})} {(\lambda  \cdot {\mathbf{x}})}}} \right.
 \kern-\nulldelimiterspace} {(\lambda  \cdot {\mathbf{x}})}}} \right] \in {K}{S}^2 [{\mathbb{R}}^n ]$ and converges strongly to $\delta ({\mathbf{x}}).)$   Thus, the finitely additive set function defined on the Borel sets (Feynman kernel \cite{FH}): (with $m=1$ and $\hbar=1$)
\[
\mathbb{K}_{\mathbf{f}} [t,{\mathbf{x}}\,;\,s,B] = \int_B {\left( {2\pi i(t - s)} \right)^{ - n/2} \exp \{ i{{\left| {{\mathbf{x}} - {\mathbf{y}}} \right|^2 } \mathord{\left/
{\vphantom {{\left| {{\mathbf{x}} - {\mathbf{y}}} \right|^2 } {2(t - s)}}} \right.
 \kern-\nulldelimiterspace} {2(t - s)}}\} d\la_3({\mathbf{y}})} 
\]
is in ${K}{S}^2 [{\mathbb{R}}^n ] $ and $
\left\| {\mathbb{K}_{\mathbf{f}} [t,{\mathbf{x}}\,;\,s,B]} \right\|_{{{KS}}}  \leqslant 1$, while $\left\| {\mathbb{K}_{\mathbf{f}} [t,{\mathbf{x}}\,;\,s,B]} \right\|_\mathfrak{M}  = \infty $ (the total variation norm) and  
\[
\mathbb{K}_{\mathbf{f}} [t,{\mathbf{x}}\,;\,s,B] = \int_{{{\R}}^3 } {\mathbb{K}_{\mathbf{f}} [t,{\mathbf{x}}\,;\,\tau ,d\la_3({\mathbf{z}})]\mathbb{K}_{\mathbf{f}} [\tau ,{\mathbf{z}}\,;\,s,B]}, \: \:{\text {(HK-integral)}}. 
\]
\begin{Def}Let ${\mathbf{P}}_n  = \{ t_0 ,\tau _1 ,t_1 ,\tau _2 , \cdots ,\tau _n ,t_n \}$ be a HK-partition for a function $\delta_n(s), \; s\in [0,t] $ for each $n$, with $\lim _{n \to \infty } \Delta \mu _n  = 0$  (mesh).  Set $\Delta t_j  = t_j  - t_{j - 1}, \tau _0  = 0$ and, for $
\psi  \in {K}{S}^2 [{\mathbb{R}}^n ] $, define
\beqa
 \int_{{\mathbf{R}}^{n[0,t]} } {\mathbb{K}_{\mathbf{f}} [{\mathcal{D}}_\lambda  {\mathbf{x}}{\text{(}}\tau ){\text{ ; }}{\mathbf{x}}{\text{(}}0{\text{)}}]}  = e^{ - \lambda t} \sum\limits_{k = 0}^{\left[\kern-0.15em\left[ {\lambda t} 
 \right]\kern-0.15em\right]}
 {\frac{{\left( {\lambda t} \right)^k }}
{{k!}}\left\{ {\prod\limits_{j = 1}^k {\int_{{\mathbf{R}}^n } {\mathbb{K}_{\mathbf{f}} [t_j ,{\mathbf{x}}{\text{(}}\tau _j {\text{)}}\,;\,t_{j - 1} ,d{\mathbf{x}}{\text{(}}\tau _{j - 1} {\text{)}}]} } } \right\}}, 
\eeqa
and 
\beqn
\: \: \: \: \: \: \: \: \: \: \int_{{\mathbf{R}}^{n[0,t]} } {\mathbb{K}_{\mathbf{f}} [{\mathcal{D}}{\mathbf{x}}{\text{(}}\tau ){\text{ ; }}{\mathbf{x}}{\text{(}}0{\text{)}}]\psi [{\mathbf{x}}{\text{(}}0{\text{)}}]}  = \lim _{\lambda  \to \infty } \int_{{\mathbf{R}}^{n[0,t]} } {\mathbb{K}_{\mathbf{f}} [{\mathcal{D}}_\lambda  {\mathbf{x}}{\text{(}}\tau ){\text{ ; }}{\mathbf{x}}{\text{(}}0{\text{)}}]\psi [{\mathbf{x}}{\text{(}}0{\text{)}}]} 
\eeqn
whenever the limit exists.
\end{Def}
\begin{rem}
In the above definition we have used the Poisson process.  This is not accidental but appears naturally from a physical analysis of the information that is knowable in the micro-world (see \cite{GZ}). This approach also provides the mathematical foundations for Feynman's sum over all paths  theory (see Section 7.7 in \cite{GZ}).   It has been suggested by Kolokoltsov \cite{KO} that such jump processes provide another way to give meaning to Feynman diagrams. 
 \end{rem}
By Corollary 2.8 (with $\mcB=L^2[{\mathbb{R}}^n ]$), ${K}{S}^{2} [{\mathbb{R}}^n ]$ is closed under convolution, so that the following is elementary.  
\begin{thm}The function $\psi ({\mathbf{x}}) = 1 \in {K}{S}^{2} [{\mathbb{R}}^n ] $ and 
\beqa
\int_{{\mathbb{R}}^{n[s,t]} } {\mathbb{K}_{\mathbf{f}} [{\mathcal{D}}{\mathbf{x}}{\text{(}}\tau ){\text{ ; }}{\mathbf{x}}{\text{(}}s{\text{)}}]}  = \mathbb{K}_{\mathbf{f}} [t,{\mathbf{x}}\,;\,s,{\mathbf{y}}] = \tfrac{1}
{{\sqrt {[2\pi i(t - s)]^n} }}\exp \{ i{{\left| {{\mathbf{x}} - {\mathbf{y}}} \right|^2 } \mathord{\left/
 {\vphantom {{\left| {{\mathbf{x}} - {\mathbf{y}}} \right|^2 } {2(t - s)}}} \right.
 \kern-\nulldelimiterspace} {2(t - s)}}\}. 
\eeqa
\end{thm}
\begin{rem}
The above result is what Feynman was trying to obtain without the appropriate space.   When a potential is present, one uses the standard perturbation methods (i.e., Trotter-Kato theorems).       A more general (sum over paths) result, that covers all application areas can be found in \cite{GZ}. 

It is clear that the position operator $\bf x$, and the momentum operator $\bf p$ have closed densely defined extensions to $KS^2$.  Treating the Fourier transform as an unitary operator, it has a bounded (unitary) extension to $KS^2$ so that $\bf x$ and $\bf p$ are still canonically conjugate pairs.  Thus, both the Heisenberg and Schr{\"o}dinger theories also have natural formulations on $KS^2$.  It is in this sense that we say that $KS^2$ is the most natural Hilbert space for quantum mechanics.
\end{rem}

If we replace the Feynman kernel by the heat kernel, we have:
\[
\mathbb{K}_{\mathbf{h}} [t,{\mathbf{x}}\,;\,s,B] = \int_B {\left( {2\pi (t - s)} \right)^{ - n/2} \exp \{ -{{\left| {{\mathbf{x}} - {\mathbf{y}}} \right|^2 } \mathord{\left/
{\vphantom {{\left| {{\mathbf{x}} - {\mathbf{y}}} \right|^2 } {2(t - s)}}} \right.
 \kern-\nulldelimiterspace} {2(t - s)}}\} d\la_n({\mathbf{y}})} 
\]
is in ${K}{S}^2 [{\mathbb{R}}^n ] $ and $
\left\| {\mathbb{K}_{\mathbf{h}} [t,{\mathbf{x}}\,;\,s,\R^n]} \right\|_{{{KS}}}  = 1$  and  
\[
\mathbb{K}_{\mathbf{h}} [t,{\mathbf{x}}\,;\,s,B] = \int_{{{\R}}^n } {\mathbb{K}_{\mathbf{h}} [t,{\mathbf{x}}\,;\,\tau ,d\la_n({\mathbf{z}})]\mathbb{K}_{\mathbf{h}} [\tau ,{\mathbf{z}}\,;\,s,B]}, \: \:{\text {(HK-integral)}}. 
\]
\begin{thm}For the function $\psi ({\mathbf{x}}) \equiv 1 \in {K}{S}^{2} [{\mathbb{R}}^n ] $, we have 
\beqa
\int_{{\mathbb{R}}^{n[s,t]} } {\mathbb{K}_{\mathbf{h}} [{\mathcal{D}}{\mathbf{x}}{\text{(}}\tau ){\text{ ; }}{\mathbf{x}}{\text{(}}s{\text{)}}]}  = \mathbb{K}_{\mathbf{h}} [t,{\mathbf{x}}\,;\,s,{\mathbf{y}}] = \tfrac{1}
{{\sqrt {[2\pi (t - s)]^n} }}\exp \{ -{{\left| {{\mathbf{x}} - {\mathbf{y}}} \right|^2 } \mathord{\left/
 {\vphantom {{\left| {{\mathbf{x}} - {\mathbf{y}}} \right|^2 } {2(t - s)}}} \right.
 \kern-\nulldelimiterspace} {2(t - s)}}\}. 
\eeqa
\end{thm}
This result implies that all known results for the Weiner path integral also have extensions to $KS^2[\R^n]$, with initial data in $HK[\R^n]$.  Furthermore, the strong continuity of the semigroup generating the heat equation means that the integral can still be concentrated on the space of continuous paths.	   
\subsection{Examples}
If we treat $K\left[ {{\mathbf{x}},\,{\text{ }}t;\;{\mathbf{y}},\,{\text{ }}s} \right]$  as the kernel for an operator acting on good initial data, then a partial solution has been obtained by a number of workers. (See \cite{GZ} for references to all the important contributions in this direction.)  The standard method is to compute the Wiener path integral for the problem under consideration and then use analytic continuation in the mass to provide a rigorous meaning for the Feynman path integral. The standard reference is Johnson and Lapidus \cite{JL}.
The following example provides a path integral representation for a problem that cannot be solved using analytic continuation via a Gaussian kernel. (For the general non-Gaussian case, see \cite{GZ}.)  It is shown that, if the vector $\bf{A}$ is constant, $\mu =mc/\hbar$, and $\boldsymbol{\beta}$  is the standard beta matrix of relativistic quantum theory, then the solution to the square-root equation for a spin $1/2$ particle: 
\[
i\hbar {{\partial \psi ({\mathbf{x}},t)} \mathord{\left/
 {\vphantom {{\partial \psi ({\mathbf{x}},t)} {\partial t}}} \right.
 \kern-\nulldelimiterspace} {\partial t}} = \left\{ {{\boldsymbol{\beta}} \sqrt {c^2 \left( {{\mathbf{p}} - \tfrac{e}
{c}{\mathbf{A}}} \right)^2  + m^2 c^4 } } \right\}\psi ({\mathbf{x}},t),{\text{  }}\psi ({\mathbf{x}},0) = \psi _0 ({\mathbf{x}}),
\]
is  given by:
$$
\psi ({\mathbf{x}},t) = {\mathbf{U}}[t,0]\psi _0 ({\mathbf{x}}) = \int\limits_{{\mathbb{R}}^3 } {\exp \left\{ {\frac{{ie}}
{{2\hbar c}}\left( {{\mathbf{x}} - {\mathbf{y}}} \right) \cdot {\mathbf{A}}} \right\}{\mathbf{K}}\left[ {{\mathbf{x}},t\,;\,{\mathbf{y}},0} \right]\psi _0 ({\mathbf{y}})d{\mathbf{y}}},
$$
where
\[
{\mathbf{K}}\left[ {{\mathbf{x}},t\,;\,{\mathbf{y}},0} \right] = \frac{{ict\mu ^2 \beta }}
{{4\pi }}\left\{ {\begin{array}{*{20}c}
   {\tfrac{{ - H_2^{(1)} \left[ {\mu \left( {c^2 t^2  - ||{\kern 1pt} {\mathbf{x}} - {\mathbf{y}}{\kern 1pt} ||^2 } \right)^{1/2} } \right]}}
{{\left[ {c^2 t^2  - ||{\kern 1pt} {\mathbf{x}} - {\mathbf{y}}{\kern 1pt} ||^2 } \right]}}{\text{, }}\;ct <  - ||{\kern 1pt} {\mathbf{x}}-{\mathbf{y}}{\kern 1pt} ||,}  \\
   {\tfrac{{ - 2iK_2 \left[ {\mu \left( {||{\kern 1pt} {\mathbf{x}} - {\mathbf{y}}{\kern 1pt} ||^2  - c^2 t^2 } \right)^{1/2} } \right]}}
{{\pi \left[ {||{\kern 1pt} {\mathbf{x}} - {\mathbf{y}}{\kern 1pt} ||^2  - c^2 t^2 } \right]}},{\text{ }}\;c\left| t \right| < \,||{\kern 1pt} {\mathbf{x}}-{\mathbf{y}}{\kern 1pt} ||,}  \\
   {\tfrac{{H_2^{(2)} \left[ {\mu \left( {c^2 t^2  - ||{\kern 1pt} {\mathbf{x}} - {\mathbf{y}}{\kern 1pt} ||^2 } \right)^{1/2} } \right]}}
{{\left[ {c^2 t^2  - ||{\kern 1pt} {\mathbf{x}} - {\mathbf{y}}{\kern 1pt} ||^2 } \right]}},{\text{ }}\;ct > \,||{\kern 1pt} {\mathbf{x}}-{\mathbf{y}}{\kern 1pt} ||.}  \\

 \end{array} } \right.
\]
The function $K_2 (\, \cdot \,)$ is a modified Bessel function of the third kind of second order, while $H_2^{(1)}, \; H_2^{(2)}$ are Hankel functions (see Gradshteyn and Ryzhik \cite{GRRZ}). Thus, we have a kernel that is far from the standard form.  This example can be found in \cite{GZ}, where we only considered the kernel for the Bessel function term.  In that case, it was shown that, under appropriate conditions, this term will reduce to the free-particle Feynman kernel and, if we set $\mu=0$, we get the kernel for a (spin $1/2$) massless particle.    
\subsection{The Navier-Stokes Problem}
In this section, we use $SD^2[\R^3]$ to provide the strongest possible a priori estimate for the nonlinear term of the classical Navier-Stokes equation.
\subsubsection{Introduction}
Let ${[L^2({\mathbb{R}}^3)]^3}$ be the Hilbert space of square integrable functions on ${\mathbb {R}}^3$,  let ${\mathbb {H}}[ {\mathbb {R}}^3 ]$ be the completion of the  set of functions in $\left\{ {{\bf{u}} \in \mathbb {C}_0^\infty  [ {\mathbb {R}}^3 ]^3 \left. {} \right|\,\nabla  \cdot {\bf{u}} = 0} \right\}$ which vanish at infinity with respect to the inner product of ${[L^2({\mathbb{R}}^3)]^3 }$.  The classical Navier-Stokes initial-value problem (on $ \mathbb{R}^3 {\text{ and all }}T > 0$) is to find a  function ${\mathbf{u}}:[0,T] \times {\mathbb {R}}^3  \to \mathbb{R}^3$ and $p:[0,T] \times {\mathbb {R}}^3  \to \mathbb{R}$ such that
\beqn\lb{ns1}
\begin{gathered}
  \partial _t {\mathbf{u}} + ({\mathbf{u}} \cdot \nabla ){\mathbf{u}} - \nu \Delta {\mathbf{u}} + \nabla p = {\mathbf{f}}(t){\text{ in (}}0,T) \times {\mathbb {R}}^3 , \hfill \\
  {\text{                              }}\nabla  \cdot {\mathbf{u}} = 0{\text{ in (}}0,T) \times {\mathbb {R}}^3 {\text{ (in the weak sense),}} \hfill \\
    {\text{                              }}{\mathbf{u}}(0,{\mathbf{x}}) = {\mathbf{u}}_0 ({\mathbf{x}}){\text{ in }}{\mathbb {R}}^3. \hfill \\ 
\end{gathered} 
\eeqn
The equations describe the time evolution of the fluid velocity ${\mathbf{u}}({\mathbf{x}},t)$ and the pressure $p$ of an incompressible viscous homogeneous Newtonian fluid with constant viscosity coefficient $\nu $ in terms of a given initial velocity ${\mathbf{u}}_0 ({\mathbf{x}})$ and given external body forces ${\mathbf{f}}({\mathbf{x}},t)$.  

Let $\mathbb{P}$ be the (Leray) orthogonal projection of 
$(L^2 [ {\mathbb {R}}^3 ])^3$ 
onto ${{\mathbb{H}}}[ {\mathbb {R}}^3]$ and define the Stokes operator by:  $ {\bf{Au}} = : -\mathbb{P} \Delta {\bf{u}}$, 
for ${\bf{u}} \in D({\bf{A}}) \subset {\mathbb{H}}^{2}[ {\mathbb {R}}^3]$, the domain of ${\bf{A}}$.      If we apply $\mathbb{P}$ to equation (\rf{ns1}), with 
${{B}}({\mathbf{u}},{\mathbf{u}}) = \mathbb{P}({\mathbf{u}} \cdot \nabla ){\mathbf{u}}$, we can recast it into the standard form:
\beqn \lb{ns2}
\begin{gathered}
  \partial _t {\mathbf{u}} =  - \nu {\mathbf{Au}} - {{B}}({\mathbf{u}},{\mathbf{u}}) + \mathbb{P}{\mathbf{f}}(t){\text{ in (}}0,T) \times \R^3 , \hfill \\
  {\text{                              }}{\mathbf{u}}(0,{\mathbf{x}}) = {\mathbf{u}}_0 ({\mathbf{x}}){\text{ in }}\R^3, \hfill \\ 
\end{gathered} 
\eeqn
where the orthogonal complement of ${\Ha} $ relative to $\{{L}^{2}(\R^3)\}^3, \;  \{ {\mathbf{v}}\,:\;{\mathbf{v}} = \nabla q,\;q \in \Ha^1[\R^3] \}$, is used to eliminate the pressure term (see \cite{GA} or [\cite{SY}, \cite{T1},\cite{T2}). 
\begin{Def}  We say that a velocity vector field in $\R^3$ is \underline{\rm reasonable} if for  $0 \le t<\iy$, there is a continuous function $m(t)>0$, depending only on $t$ and a constant $M_0$, which may depend on ${\bf u}_0$ and $f$, such that 
\[
0 < m(t) \leqslant \left\| { {\mathbf{u}}(t)} \right\|_{{\mathbb{H}}} \le M_0.
\] 
\end{Def}
The above definition formalizes the requirement that the fluid has nonzero but bounded positive definite energy.  However, this condition still allows the velocity to approach zero at infinity in a weaker norm.
\subsubsection{The Nonlinear Term: A Priori Estimates}
The difficulty in proving the existence and uniqueness of global-in-time strong solutions for equation (\rf{ns2}) is directly linked to the problem of getting good a priori estimates for the nonlinear term ${{B}}({\mathbf{u}},{\mathbf{u}})$.  For example, using standard methods on $\Ha$, the following estimates are known.  If ${\mathbf{u}}, {\mathbf{v}} \in D({\mathbf{A}})$, a typical bound in the $\Ha$ norm for the nonlinear term in equation (\rf{ns2}) can be found in Sell and You \cite{SY} (see page 366):  
\[
\max \left\{ {{{\left\| {B({\mathbf{u}},{\mathbf{v}})} \right\|}_\mathbb{H}},\;{{\left\| {B({\mathbf{v}},{\mathbf{u}})} \right\|}_\mathbb{H}}} \right\} \leqslant {C_0}{\left\| {{{\mathbf{A}}^{5/8}}{\mathbf{u}}} \right\|_\mathbb{H}}{\left\| {{{\mathbf{A}}^{5/8}}{\mathbf{v}}} \right\|_\mathbb{H}}.
\]
In this section, we show how $SD^2[\R^3]$ allows us to obtain the best possible a priori estimates.  Let $\Ha_{sd}$ be the closure of $\Ha \cap SD^2[\R^3]$ in the $SD^2$ norm.  
\begin{thm}   Suppose that the reasonable vector field ${\bf u}({\bf x},t) \in { SD}^2[\R^3] \cap D({\bf{A}})$ satisfies (\rf{ns2}) and  $\bf A$ is the Stokes operator.  Then 
\begin{enumerate}
\item 
\beqn\lb{ns3}
 {\left\langle { \nu{\bf{A}}{\mathbf{u}},{\mathbf{u}}} \right\rangle _{{\mathbb{H}}_{sd} }} = 3\left\| { {\mathbf{u}} } \right\|_{ {\mathbb{H}}_{sd} }^2.
\eeqn
\item There exists a constants $M_1, M_2; M_i =M_i({\bf u}_0,{\bf{f}}))>0$, such that
\beqn\lb{ns4}
 \left| {\left\langle {B({\mathbf{u}},{\mathbf{u}}),{\mathbf{u}}} \right\rangle _{{\mathbb{H}}_{sd} }} \right| \le M_1 \left\| { {\mathbf{u}} } \right\|_{ {\mathbb{H}}_{sd} }^3 
\eeqn
\item and
\beqn\lb{ns5}
 max\{ \left\| {{{B}}({\mathbf{u}},{\mathbf{v}})} \right\|_{{\mathbb{H}}_{sd}}, \ \left\| {{{B}}({\mathbf{v}},{\mathbf{u}})} \right\|_{{\mathbb{H}}_{sd}} \} \leqslant M_2 \left\| {\mathbf{u}} \right\|_{{\mathbb{H}}_{sd}} \left\| {\mathbf{v}} \right\|_{{\mathbb{H}}_{sd}}. 
\eeqn
\end{enumerate}
\end{thm}
\begin{proof} From the definition of the inner product, for $\Ha_{sd}[\R^3]$ we have
\beqa
{\left\langle {\nu {\bf{A}}{\mathbf{u}},{\mathbf{u}}} \right\rangle _{{\mathbb{H}}_{sd} }}=\nu
 \sum\nolimits_{m = 1}^\infty  {t_m } \left[ {\int_{\mathbb{R}^3 } {{\mathcal{E}}_m ({\mathbf{x}})\cdot {\bf{A}}{\mathbf{u}}({\mathbf{x}})d{\la_3(\bf{x})}} } \right] \overline{\left[ {\int_{\mathbb{R}^3 } {{\mathcal{E}}_m ({\mathbf{y}})\cdot {\mathbf{u}}({\mathbf{y}})d{\la_3(\bf{y})}} } \right]}.  
\eeqa
Using the fact that ${\bf{u}} \in D({\bf A})$, it follows that
\[
{\int_{\mathbb{R}^3 } {{\mathcal{E}}_m ({\mathbf{y}}) \cdot {\partial _{y_j }^2 }{\mathbf{u}}({\mathbf{y}})d{\la_3(\bf{y})}} }= {\int_{\mathbb{R}^3 } {\partial _{y_j }^2} {{\mathcal{E}}_m ({\mathbf{y}}) \cdot{\mathbf{u}}({\mathbf{y}})d{\la_3(\bf{y})}} }=(i)^2 {\int_{\mathbb{R}^3 } {{\mathcal{E}}_m ({\mathbf{y}}) \cdot{\mathbf{u}}({\mathbf{y}})d{\la_3(\bf{y})}} }.  
\]
Using this in the above equation and summing on $j$, we have (${\bf A }=-\mathbb{P} \De$) 
\[ 
{\int_{\mathbb{R}^3 } {{\mathcal{E}}_m ({\mathbf{y}}) \cdot {\bf A }{\mathbf{u}}({\mathbf{y}})d{\la_3(\bf{y})}} } = 3{\int_{\mathbb{R}^3 } {{\mathcal{E}}_m ({\mathbf{y}}) \cdot{\mathbf{u}}({\mathbf{y}})d{\la_3(\bf{y})}} }.
\]
It follows that
\beqa
\begin{gathered}
{\left\langle { {\bf{A}}{\mathbf{u}},{\mathbf{u}}} \right\rangle _{{\mathbb{H}}_{sd} }}\hfill \\
=3 \sum\nolimits_{m = 1}^\infty  {t_m } \left[ {\int_{\mathbb{R}^3 } {{\mathcal{E}}_m ({\mathbf{x}})\cdot {\mathbf{u}}({\mathbf{x}})d{\la_3(\bf{x})}} } \right] \overline{\left[ {\int_{\mathbb{R}^3 } {{\mathcal{E}}_m ({\mathbf{y}})\cdot {\mathbf{u}}({\mathbf{y}})d{\la_3(\bf{y})}} } \right]} \hfill \\
 =3\left\| {{\mathbf{u}} } \right\|_{ {\mathbb{H}}_{sd} }^2.  \hfill \\
 \end{gathered}
\eeqa
This proves (\rf{ns3}).  To prove (\rf{ns4}),  let
\[
b\left( {{\mathbf{u}},{\mathbf{v}},{{\mathcal{E}}_m}} \right) = \int_{{\mathbb{R}^3}} {\left( {{\mathbf{u}}({\mathbf{x}}) \cdot \nabla }{\mathbf{v}}({\mathbf{x}}) \right) \cdot {{\mathcal{E}}_m}({\mathbf{x}})d{\lambda _3}({\mathbf{x}})} 
\]
and define the vector ${\bf{I}}$ by ${\bf{I}}=[1,1,1]^t$.  We start with integration by parts and $\nabla  \cdot {\mathbf{u}} = 0$, to get
\[
b\left( {{\mathbf{u}},{\mathbf{v}},{{\mathcal{E}}_m}} \right) =  - b\left( {{\mathbf{u}},{{\mathcal{E}}_m},{\mathbf{v}}} \right) =  - i\int_{{\mathbb{R}^3}} {\left( {{\mathbf{u}}({\mathbf{x}}) \cdot {\mathbf{I}}} \right)\left( {{{\mathcal{E}}_m}({\mathbf{x}}) \cdot {\mathbf{v}}({\mathbf{x}})} \right)d{\lambda _3}({\mathbf{x}})} .
\]
 From the above equation, we have ($m \leftrightarrow (k,i)$)    
\[
\begin{gathered}
  \left| {b\left( {{\mathbf{u}},{\mathbf{v}},{{\mathcal{E}}_m}} \right)} \right| \leqslant \sqrt{3}\int_{{\mathbb{R}^3}} {{{\left| {{\mathbf{u}}({\mathbf{x}})} \right|}}{{\left| {{\mathbf{v}}({\mathbf{x}})} \right|}}d{\lambda _3}({\mathbf{x}})\mathop {\sup }\limits_{k} } \left\| {{{\mathcal{E}}_m}} \right\|_\iy \hfill \\
   \leqslant C_1\left\| {\mathbf{u}} \right\|_\mathbb{H}\left\| {\mathbf{v}} \right\|_\mathbb{H}. \hfill \\ 
\end{gathered} 
\] 
We also have:
\[
\left| {\int_{{\mathbb{R}^3}} {{\mathbf{w}}({\mathbf{x}}) \cdot {{\mathcal{E}}_m}({\mathbf{x}})d{\lambda _3}({\mathbf{x}})} } \right| \leqslant C_2\left\| {\mathbf{w}} \right\|_\mathbb{H}.
\]
If we combine the last two results, we get that:
\beqn\lb{ns6}
\begin{gathered}
  \left| {\left\langle {B({\mathbf{u}},{\mathbf{v}}),{\mathbf{w}}} \right\rangle _{\mathbb{H}_{sd} } } \right|  \hfill \\
   \leqslant \sum\limits_{m = 1}^\infty  {{t_m}\left| {b\left( {{\mathbf{u}},{\mathbf{v}},{{\mathcal{E}}_m}} \right)}  \right|\overline {\left| {\int_{{\mathbb{R}^3}} {{\mathbf{w}}({\mathbf{y}}) \cdot {{\mathcal{E}}_m}({\mathbf{y}})d{\lambda _3}({\mathbf{y}})} } \right|} }  \hfill \\
   \leqslant C\left\| {\mathbf{u}} \right\|_\mathbb{H}\left\| {\mathbf{v}} \right\|_\mathbb{H}\left\| {\mathbf{w}} \right\|_\mathbb{H}. \hfill \\ 
\end{gathered} 
\eeqn
Since ${\bf{u}}, {\bf{v}}, {\bf{w}}$ are reasonable velocity vector fields, there is a constant  ${M_1}$  depending on ${\bf{u}}_0, {\bf{v}}_0, {\bf{w}}_0$ and $f$,  such that  
\beqn\lb{ns7}
C\left\| {\mathbf{u}} \right\|_\mathbb{H}\left\| {\mathbf{v}} \right\|_\mathbb{H}\left\| {\mathbf{w}} \right\|_\mathbb{H}  \le {M_1} \left\| {\mathbf{u}} \right\|_{\mathbb{H}_{sd} }  \left\| {\mathbf{v}} \right\|_{\mathbb{H}_{sd} }  \left\| {\mathbf{w}} \right\|_{\mathbb{H}_{sd} }.
\eeqn  
If ${\mathbf{w}}={\mathbf{v}}={\mathbf{u}}$, we have that:
\[
  \left| {\left\langle {B({\mathbf{u}},{\mathbf{u}}),{\mathbf{u}}} \right\rangle _{\mathbb{H}_{sd} } } \right|  
   \leqslant {M_1}\left\| {\mathbf{u}} \right\|_{\mathbb{H}_{sd} }^3.  
\]
This proves (\rf{ns4}).  The proof of (\rf{ns5}) is a straight forward application of (\rf{ns6}) and (\rf{ns7}).
\end{proof}

\section*{\textbf{Conclusion}}
In this survey we have constructed a number of new classes of Banach spaces ${{KS}}^p, \; 1 \le p \le \iy$, ${{SD}}^p, \; 1 \le p \le \iy$, ${{\mcZ}}^p, \; 1 \le p \le \iy$, ${{\mcZ}}^{-p}, \; 1 \le p \le \iy$, $BMO^w$ and $BMO^{-1}$.    These spaces are of particular interest because they contain the HK-integrable functions.  (They also contain all functions that are integrable via any of the classical integrals.) The ${{KS}}^p$ and ${{SD}}^p$ class contain the standard $L^p$ spaces as dense compact embeddings.  They also contain the test functions $\mcD[\R^n]$, so that $\mcD'[\R^n]$ is a continuous embedding into the corresponding dual space.  The ${{SD}}^p$ class has the remarkable property that ${\left\| {{D^\alpha }f} \right\|_{SD}} = {\left\| f \right\|_{SD}}$, for every index $\al$.

The space $BMO^w$ and the families ${{\mcZ}}^{p}, \; 1 \le p \le \iy$ extend the space $BMO$ (i.e., functions of bounded mean oscillation).
The space $BMO^{-w}$ and the families ${{\mcZ}}^{-p}, \; 1 \le p \le \iy$ extend the related space $BMO^{-1}$.

In the analytical theory of Markov processes, it is well-known that, in general, the semigroup $T(t)$ associated with the process is not strongly continuous on ${\mathbb{C}}_b [\mathbb{R}^n ]$, the space of bounded continuous functions or ${\mathbb{UBC}}[\mathbb{R}^n ]$, the bounded uniformly continuous functions.   We have shown that the weak generator defined by the mixed locally convex topology on ${\mathbb{C}}_b [\mathbb{R}^n ]$ is a strong generator on ${K}{S}^p [\mathbb{R}^n ]$ (e.g., $T(t)$ is strongly continuous on ${{KS}}^p[\mathbb{R}^n ]$ for $1 \le p \le \infty$). 

We also have used ${{KS}}^2$ to construct the free-particle path integral in the manner originally intended by  Feynman.  It is shown in \cite{GZ} that ${{KS}}^2$ has a claim as the natural representation space for the Feynman formulation of quantum theory in that it allows representations for both the Heisenberg and Schr\"{o}dinger representations for quantum mechanics, a property not shared by ${{L}}^2$.     
 
We also have used ${{SD}}^2$ to to provided the strongest possible a priori bounds for the nonlinear term of the classical Navier-Stokes equation.

\end{document}